\documentclass[12pt]{amsart}

\usepackage{fullpage}
\usepackage{amssymb}
\usepackage{microtype}
\usepackage{xcolor}
\usepackage{mathrsfs}
\usepackage{tikz}
\usepackage{hyperref}

\usetikzlibrary{cd}

\usepackage{bm}

\title{Algebraic dependence and Milnor K-theory}
\author{Adam Topaz}
\address{
    Mathematical and Statistical Sciences \\
    University of Alberta \\
    632 Central Academic Building \\
    Edmonton AB T6G 2G1 \\
    Canada }
\subjclass{12G99, 12J10, 19D45}
\email{topaz@ualberta.ca}
\urladdr{\url{https://adamtopaz.com/}}
\date{\today}
\thanks{Research supported by the author's NSERC discovery grant.}
\keywords{Algebraic dependence, Milnor K-theory, valuation theory, anabelian geometry}
\date{Last updated: \today}

\newcommand{\cl}{\operatorname{cl}}
\newcommand{\Sup}{\operatorname{Sup}}

\newcommand{\KM}{\operatorname{K}^\mathrm{M}}

\newcommand{\T}{\operatorname{T}}

\newcommand{\Zbb}{\mathbb{Z}}
\newcommand{\Qbb}{\mathbb{Q}}
\newcommand{\Cbb}{\mathbb{C}}
\newcommand{\Gbb}{\mathbb{G}}

\newcommand{\smin}{\smallsetminus}
\newcommand{\trdeg}{\operatorname{trdeg}}
\newcommand{\acl}{\operatorname{acl}}
\newcommand{\Hom}{\operatorname{Hom}}
\newcommand{\Gal}{\operatorname{Gal}}

\newcommand{\Urm}{\mathrm{U}}
\newcommand{\Ocal}{\mathcal{O}}
\newcommand{\Bcal}{\mathcal{B}}
\newcommand{\mfrak}{\mathfrak{m}}

\newcommand{\Ascr}{\mathscr{A}}

\newcommand{\Vscr}{\mathscr{V}}

\newcommand{\Zcal}{\mathcal{Z}}

\newcommand{\Ccal}{\mathcal{C}}
\newcommand{\Gcal}{\mathcal{G}}
\newcommand{\Rcal}{\mathcal{R}}
\newcommand{\Dcal}{\mathcal{D}}
\newcommand{\Ical}{\mathcal{I}}
\newcommand{\Ecal}{\mathcal{E}}
\newcommand{\Hcal}{\mathcal{H}}
\newcommand{\Ucal}{\mathcal{U}}

\newcommand{\Kcal}{\mathcal{K}}

\renewcommand{\tilde}{\widetilde}
\renewcommand{\bar}{\overline}
\renewcommand{\epsilon}{\varepsilon}
\renewcommand{\phi}{\varphi}

\newtheorem{theorem}{Theorem}[section]
\newtheorem{lemma}[theorem]{Lemma}

\newtheorem{proposition}[theorem]{Proposition}

\newtheorem{fact}[theorem]{Fact}

\newtheorem*{maintheorem*}{Main Theorem}

\theoremstyle{definition}

\begin{document}
\maketitle

\begin{abstract}
  This paper shows that algebraic (in)dependence is encoded in Milnor K-theory of fields.
  As an application, we show that the isomorphism type of a field is determined by its Milnor K-theory, up to purely inseparable extensions, in most situations.
\end{abstract}

\setcounter{tocdepth}{1}
\tableofcontents

\section{Introduction}

Let $K$ be a field.
The Milnor K-theory of $K$ has a very simple definition:
\[ \KM_{*}(K) := \frac{\T_{*}(K^{\times})}{ \langle x \otimes y \ | \ x + y = 1 \rangle } \]
where $\T_{*}(K^{\times})$ denotes the tensor algebra of the $\Zbb$-module $K^{\times}$, and the two-sided ideal $\langle x \otimes y \ | \ x + y = 1 \rangle$ consists of the so-called \emph{Steinberg relations}.

In degree one, we have the multiplicative group, $\KM_{1}(K) = K^{\times}$, while the ring structure of $\KM_{*}(K)$ involves the additive structure of $K$ as well.
It is natural to ask whether the \emph{field} $K$ itself is determined (up-to isomorphism) by $\KM_{*}(K)$.
This question was considered in~\cite{zbMATH05635168,zbMATH07463742}, focusing mostly on finitely-generated field extensions and eventually relying on the so-called \emph{fundamental theorem of projective geometry} to reconstruct the fields in question.

In this paper, we investigate this question for fields which do not necessarily satisfy any finiteness conditions, and we obtain the following main result.
\begin{maintheorem*}
  Let $K$ be any field whose absolute transcendence degree is at least $5$.
  Then the isomorphism type of $K$ is determined, up to purely inseparable extensions, by the $\Qbb$-algebra $\Qbb \otimes_{\Zbb} \KM_{*}(K)$.
\end{maintheorem*}
See Theorem~\ref{theorem:main_absolute_Kcal} for the precise statement.
We also prove a similar \emph{relative} result for relatively algebraically closed field extensions of sufficiently large transcendence degree in Theorem~\ref{theorem:main_relative_Kcal}.
Note that the theorem above imposes no additional restrictions on $K$ besides the bound on the absolute transcendence degree.
For example, this theorem applies to any sufficiently large \emph{algebraically closed} field.

Since we wish to work with fields whose multiplicative group may even be \emph{divisible}, it is important to work with $\Qbb \otimes_{\Zbb} \KM_{*}(K)$ as opposed to $\KM_{*}(K)$ itself.
More precisely, if $K$ is radically closed, then the quotient $K^{\times}/\mathrm{torsion}$ is already a $\Qbb$-vector space, and thus $\Qbb^{\times}$ provides a source of indeterminacy for $\KM_{*}(K)/\langle \text{torsion} \rangle$.
Namely, any $r \in \Qbb^{\times}$ yields an automorphism of $\KM_{*}(K)/\langle \mathrm{torsion} \rangle$ defined by $t \mapsto t^{r}$ in degree one, and such automorphisms only arise from field theory if $r$ is a power of the characteristic of $K$.
By tensoring $\KM_{*}(K)$ with $\Qbb$, we can control for such indeterminacies in our main result.

Furthermore, if $K$ is any field and $K \to K^{i}$ denotes the perfection of $K$, then the corresponding map
\[ \KM_{*}(K) \to \KM_{*}(K^{i}) \]
induces an \emph{isomorphism} after tensoring with $\Qbb$.
Thus, inseparability is an additional source of indeterminacy which must be accounted for, hence we can only expect to recover the isomorphism type of $K$ up to purely inseparable extensions when working with $\Qbb \otimes_{\Zbb} \KM_{*}(K)$.

The technical core of this work is in recovering all the information about \emph{algebraic dependence} from Milnor K-theory, see Theorem~\ref{theorem:recover_Kcal_geom_lattice}.
Once we obtain all information about algebraic dependence, our reconstruction results will follow by applying a distant cousin of the fundamental theorem of projective geometry, due to Evans-Hrushovski~\cite{zbMATH00007333,zbMATH00839199} and Gismatullin~\cite{zbMATH05350217}, based on the \emph{group-configuration theorem}.
In the case where the fields (or field extensions) in question are finitely-generated, we can instead use one of the main theorems from~\cite{zbMATH07463742} to obtain better reconstruction results.
For example, Theorem~\ref{theorem:fg_case_absolute} (which uses~\cite[Theorem 4]{zbMATH07463742} in an essential way) shows that the isomorphism type of a finitely-generated field $K$ of absolute transcendence degree $\geq 2$ is determined by $\KM_{*}(K)$ with no need to pass to inseparable extensions.

\section{Notation and preliminaries}

We will primarily work with a fixed field denoted by $K$.
In some cases we will also consider subfields of $K$, usually denoted $k$.

\subsection{Quotients of Milnor K-theory}
For a subgroup $T$ of $K^{\times}$, we write
\[ \KM_{*}(K|T) := \frac{\KM_{*}(K)}{\langle T \rangle} \]
where $\langle T \rangle$ refers to the (two-sided) ideal of $\KM_{*}(K)$ generated by $T \subset K^{\times} = \KM_{1}(K)$.
If $T$ is trivial, then we omit it from the notation to match the standard notation for Milnor K-theory: $\KM_{*}(K) := \KM_{*}(K|\{1\})$.
In the case where $T = k^{\times}$ for a subfield $k$ of $K$, we write $\KM_{*}(K|k)$ instead of $\KM_{*}(K|k^{\times})$.

As usual, we will use the notation $\{f_{1},\ldots,f_{n}\}$ to denote the product of $f_{1},\ldots,f_{n} \in K^{\times} = \KM_{1}(K)$ in $\KM_{n}(K)$, and such elements of $\KM_{n}(K)$ will be called \emph{symbols}.
A similar notation and terminology will also be used for the variants of $\KM_{*}(K)$ we consider in this paper, while ensuring that the variant being considered is clear from context.

\subsection{Duality}

For a subgroup $T$ of $K^{\times}$, we write
\[ \Kcal_{K|T} := \Qbb \otimes_{\Zbb}(K^{\times}/T). \]
If $T = \{1\}$, then we write $\Kcal_{K}$ instead of $\Kcal_{K|T}$, and if $T = k^{\times}$ for a subfield $k$ of $K$, we will write $\Kcal_{K|k}$ instead of $\Kcal_{K|k^{\times}}$.
The operation of $\Kcal_{K|T}$ will always be written \emph{additively}.

We will consider the dual
\[ \Gcal_{K} := \Hom_{\Zbb}(K^{\times},\Qbb) = \Hom_{\Qbb}(\Kcal_{K},\Qbb) \]
as a topological vector space with respect to the weak topology, where $\Qbb$ is given the discrete topology.
We have an obvious perfect pairing
\[ K^{\times} \times \Gcal_{K} \to \Qbb. \]
For a subspace $\Hcal$ of $\Gcal_{K}$, we write $\Hcal^{\perp} \subset K^{\times}$ for the orthogonal of $\Hcal$ with respect to this pairing.
For a subgroup $T$ of $K^{\times}$, we will use the notation $\Gcal_{K|T} \subset \Gcal_{K}$ for the orthogonal of $T$ with respect to this pairing.
As always, if $T = k^{\times}$ for a subfield $k$ of $K$, we write $\Gcal_{K|k}$ instead of $\Gcal_{K|k^{\times}}$.
When $T = \{1\}$ is trivial, one has $\Gcal_{K|T} = \Gcal_{K}$, so our convention of omitting $T$ from the notation in this case still works.

A subgroup $T$ of an abelian group $A$ will be called \emph{saturated} if $A/T$ is torsion-free.
For any subgroup $T$ of $K^{\times}$, the subspace $\Gcal_{K|T}$ is closed, and for a subspace $\Hcal \subset \Gcal_{K}$, the subgroup $\Hcal^{\perp}$ is saturated.
In fact, if $T$ is any subgroup of $K^{\times}$, then $\Gcal_{K|T}^{\perp}$ is the saturation of $T$ (i.e.~the smallest saturated subgroup containing $T$) and if $\Hcal \subset \Gcal_{K}$ is any subspace then $\Gcal_{K|\Hcal^{\perp}}$ is the closure of $\Hcal$.

The maps $\Hcal \mapsto \Hcal^{\perp}$ and $T \mapsto \Gcal_{K|T}$ provide a one-to-one order-reversing correspondence between the closed subspaces of $\Gcal_{K}$ and the saturated subgroups of $K^{\times}$.
We also have canonical perfect pairing
\[ \Kcal_{K|T} \times \Gcal_{K|T} \to \Qbb \]
associated to any subgroup $T$ of $K^{\times}$.
We will say so explicitly when considering orthogonals with respect to this pairings to avoid any potential confusion with the notation $(-)^{\perp}$ introduced above.

The base-change $\Qbb \otimes_{\Zbb} \KM_{*}(K|T)$ will be denoted by $\Kcal_{*}(K|T)$.
As usual, if $T = \{1\}$ then we write $\Kcal_{*}(K)$ instead of $\Kcal_{*}(K|\{1\})$ and if $T = k^{\times}$ for a subfield $k$ of $K$, then we write $\Kcal_{*}(K|k)$ instead of $\Kcal_{*}(K|k^{\times})$.
Note that for any subgroup $T$ of $K^{\times}$, one has $\Kcal_{1}(K|T) = \Kcal_{K|T}$.

\subsection{Alternating pairs}

Elements of $\Gcal_{K}$ will be considered both as $\Zbb$-linear maps $K^{\times} \to \Qbb$ and as $\Qbb$-linear maps $\Kcal_{K} \to \Qbb$.
If $f \in \Gcal_{K}$ with $T = \ker(f) \subset K^{\times}$, then we may also consider $f$ as a $\Zbb$-linear map $K^{\times}/T \to \Qbb$ and as a $\Qbb$-linear map $\Kcal_{K|T} \to \Qbb$.

A pair of elements $f,g \in \Gcal_{K}$ will be called an \emph{alternating pair} provided that
\[ f(x) \cdot g(y) = f(y) \cdot g(x) \]
whenever $x,y \in K^{\times}$ satisfy $x + y = 1$ in $K$.
We denote the associated binary relation on $\Gcal_{K}$ by $\Rcal_{K}$:
\[ \Rcal_{K}(f,g) \Longleftrightarrow \text{$f,g$ are an alternating pair.} \]
In fact, for the majority of this paper we will be working with the structure consisting of the following data (associated to various subgroups $T$ of $K^{\times}$), which we abbreviate as $\Ascr(K|T)$:
\begin{enumerate}
  \item The $\Qbb$-vector space $\Kcal_{K|T}$.
  \item The topological $\Qbb$-vector space $\Gcal_{K|T}$.
  \item The canonical pairing $\Kcal_{K|T} \times \Gcal_{K|T} \to \Qbb$.
  \item The restriction of the relation $\Rcal_{K}$ to $\Gcal_{K|T}$.
\end{enumerate}
As before, we write $\Ascr(K)$ instead of $\Ascr(K|T)$ if $T$ is trivial and $\Ascr(K|k)$ instead of $\Ascr(K|k^{\times})$ when $k$ is a subfield of $K$.

Recall that the Steinberg relations in Milnor K-theory are generated by basic tensors of the form $x \otimes y$ for $x,y \in K^{\times}$ satisfying $x + y = 1$.
Thus, the alternating condition for pairs of elements of $\Gcal_{K}$ can be tested using the product in Milnor K-theory, as the following fact summarizes.
\begin{fact}\label{fact:alternating_iff_milnor_condition}
  Let $f,g \in \Gcal_{K}$ be given and let $T \subset (\Qbb \cdot f + \Qbb \cdot g)^{\perp}$ be any subgroup.
  The following are equivalent:
  \begin{enumerate}
    \item For all $x,y \in K^{\times}/T = \KM_{1}(K|T)$ satisfying $\{x,y\} = 0$ in $\KM_{2}(K|T)$, one has
          \[ f(x) \cdot g(y) = f(y) \cdot g(x). \]
    \item For all $x,y \in \Kcal_{K|T} = \Kcal_{1}(K|T)$ satisfying $\{x,y\} = 0$ in $\Kcal_{2}(K|T)$, one has
          \[ f(x) \cdot g(y) = f(y) \cdot g(x). \]
    \item $f,g$ are an alternating pair.
  \end{enumerate}
\end{fact}
In particular, this shows that the data $\Ascr(K|T)$ is \emph{completely determined} (in a functorial manner) by the algebra $\Kcal_{*}(K|T)$.
Indeed, $\Gcal_{K|T}$ is the (weak) dual of $\Kcal_{K|T} = \Kcal_{1}(K|T)$, and the fact above shows that for $f,g \in \Gcal_{K|T}$ one has $\Rcal_{K}(f,g)$ if and only if for all $x,y \in \Kcal_{1}(K|T)$ such that $\{x,y\} = 0$ in $\Kcal_{2}(K|T)$, one has $f(x) \cdot g(y) = f(y) \cdot g(x)$.

We will borrow some notation and terminology from group theory by considering $\Rcal_{K}(-,-)$ as being analogous to the condition that two elements of a group commute.
Namely, for a closed subspace $\Hcal$ of $\Gcal_{K}$ we consider the following (closed) subspaces of $\Gcal_{K}$:
\begin{enumerate}
  \item $\Ccal_{K}(\Hcal) := \{f \in \Gcal_{K} \ | \ \forall g \in \Hcal, \ \Rcal_{K}(f,g)\}$, the $\Rcal_{K}$-\emph{centralizer} of $\Hcal$.
  \item $\Zcal_{K}(\Hcal) := \{f \in \Hcal \ | \ \forall g \in \Hcal, \ \Rcal_{K}(f,g)\}$, the $\Rcal_{K}$-\emph{centre} of $\Hcal$.
\end{enumerate}

\subsection{Valuations}

Valuations on $K$ will only be considered up-to equivalence.
Let $v,w$ be two valuations.
Our convention is that $v \le w$ means $v$ is a \emph{coarsening} of $w$.

Let $v$ be a valuation of $K$.
We shall write $\Ocal_{v}$ for the valuation ring of $v$, $\mfrak_{v}$ for the maximal ideal of $\Ocal_{v}$, $Kv$ for the residue field of $v$ and $vK$ for the value group of $v$.
The unit group $\Ocal_{v}^{\times}$ will be denoted by $\Urm_{v}$ and the principal unit group $1 + \mfrak_{v}$ will be denoted by $\Urm_{v}^{1}$.
If $k$ is a subfield of $K$, then we write $kv$ and $vk$ for the residue field and value group of the restriction of $v$ to $k$.

We define:
\[ \Ical_{v} := \Gcal_{K|\Urm_{v}}, \ \ \Dcal_{v} := \Gcal_{K|\Urm_{v}^{1}}. \]
Note that $\Ical_{v} \subset \Dcal_{v}$ and that the exact sequence
\[ 1 \to Kv^{\times} \to K^{\times}/\Urm_{v}^{1} \to vK \to 1 \]
dualizes to an exact sequence
\[ 0 \to \Ical_{v} \to \Dcal_{v} \to \Gcal_{Kv} \to 0. \]

For a subgroup $T$ of $K^{\times}$, write $Tv$ for the image of $T \cap \Urm_{v}$ in $Kv^{\times}$ and $vT$ for the image of $T$ in $vK$.
We have an exact sequence
\[ 1 \to Kv^{\times}/Tv \to K^{\times}/(T \cdot \Urm_{v}^{1}) \to vK/vT \to 1 \]
which dualizes to an exact sequence of the form
\[ 0 \to \Gcal_{K|T} \cap \Ical_{v} \to \Gcal_{K|T} \cap \Dcal_{v} \to \Gcal_{Kv|Tv} \to 0. \]
In the special case where $T = k^{\times}$ for a subfield $k$ of $K$, our notational conventions are compatible.
Namely, the natural exact sequence
\[ 1 \to Kv^{\times}/kv^{\times} \to K^{\times}/(k^{\times} \cdot \Urm_{v}^{1}) \to vK/vk \to 1 \]
dualizes to an exact sequence
\[ 0 \to \Gcal_{K|k} \cap \Ical_{v} \to \Gcal_{K|k} \cap \Dcal_{v} \to \Gcal_{Kv|kv} \to 0. \]

The subspace of $\Kcal_{K}$ generated by the image of $\Urm_{v}^{1}$ will be denoted by $\Ucal_{v}^{1}$ and the subspace generated by the image of $\Urm_{v}$ will be denoted by $\Ucal_{v}$.
Similarly, if $T$ is a subgroup of $K^{\times}$, we write $\Ucal_{v|T}^{1}$ for the image of $\Ucal_{v}^{1}$ in $\Kcal_{K|T}$ and $\Ucal_{v|T}$ for the image of $\Ucal_{v}$.
As always, when $T = k^{\times}$, we write $\Ucal_{v|k}$ and $\Ucal_{v|k}^{1}$ instead of $\Ucal_{v|k^{\times}}$ and $\Ucal_{v|k^{\times}}^{1}$.

Note that $\Ucal_{v}$ resp.~$\Ucal_{v}^{1}$ is the orthogonal of $\Ical_{v}$ resp.~$\Dcal_{v}$ with respect to the pairing $\Kcal_{K} \times \Gcal_{K} \to \Qbb$.
Similarly, $\Ucal_{v|T}$ resp.~$\Ucal_{v|T}^{1}$ is the orthogonal of $\Gcal_{K|T} \cap \Ical_{v}$ resp.~$\Gcal_{K|T} \cap \Dcal_{v}$ with respect to the pairing $\Kcal_{K|T} \times \Gcal_{K|T} \to \Qbb$.

\section{The local theory}

Our starting point is the following fundamental result.
\begin{theorem}\label{theorem:main_acl}
  Let $K$ be any field, and let $\Dcal$ be a closed subspace of $\Gcal_{K}$.
  The following are equivalent:
  \begin{enumerate}
    \item For all $f,g \in \Dcal$, one has $\Rcal_{K}(f,g)$.
    \item There exists a valuation $v$ of $K$ and a closed subspace $\Ical \subset \Dcal$ of codimension $\le 1$, such that $\Dcal \subset \Dcal_{v}$ and $\Ical \subset \Ical_{v}$.
  \end{enumerate}
\end{theorem}
Variants of this theorem have appeared in the works of Bogomolov~\cite{zbMATH00124400}, Bogomolov-Tschinkel~\cite{zbMATH02078165}, Efrat~\cite{zbMATH01489500}, Koenigsmann~\cite{zbMATH00785379}, Engler-Koenigsmann~\cite{zbMATH01135369}, the author~\cite{zbMATH06778776,https://doi.org/10.48550/arxiv.1705.01084}, and others, albeit primarily the Galois-theoretic context.
The proof of Theorem~\ref{theorem:main_acl} has now been completely formally verified using the \texttt{Lean3} interactive theorem prover~\cite{10.1007/978-3-319-21401-6_26} and its formally verified mathematics library \texttt{mathlib}~\cite{10.1145/3372885.3373824}, see~\cite{lean-acl-pairs}.
We thus omit the proof, referring instead to the references above for the key ideas and to~\cite{lean-acl-pairs} for the computer-verified proof.

The power of this theorem is in the implication $(1) \Rightarrow (2)$, while the converse is a simple consequence of the ultrametric inequality.
We will need a slightly stronger variant of the ``easy'' direction $(2) \Rightarrow (1)$, formulated as follows.
\begin{lemma}\label{lemma:Rcal_residue_compat}
  Suppose that $v$ is a valuation of $K$ and $f,g \in \Dcal_{v}$ are given.
  Let $f_{v}$ and $g_{v}$ denote the images of $f$ and $g$ in $\Gcal_{Kv}$ under the canonical map $\Dcal_{v} \to \Gcal_{Kv}$.
  Then $\Rcal_{K}(f,g)$ holds if and only if $\Rcal_{Kv}(f_{v},g_{v})$ holds.
\end{lemma}
\begin{proof}
  Suppose $\Rcal_{K}(f,g)$ holds, and let $x,y \in Kv^{\times}$ satisfy $x + y = 1$.
  We may choose lifts $\tilde x, \tilde y \in \Urm_{v}$ of $x, y$ such that $\tilde x + \tilde y = 1$.
  Thus
  \[ f_{v}(x) \cdot g_{v}(y) = f(\tilde x) \cdot g(\tilde y) = f(\tilde y) \cdot g(\tilde x) = f_{v}(y) \cdot g_{v}(y). \]

  Conversely, suppose that $\Rcal_{Kv}(f_{v},g_{v})$ holds and let $x,y \in K^{\times}$ be such that $x + y = 1$.
  We must show that
  \[ f(x) \cdot g(y) = f(y) \cdot g(x). \]
  If $v(x) > 0$ then $f(y) = g(y) = 0$ since $f,g \in \Dcal_{v}$, so the equation in question trivially holds true.
  The equation similarly holds true if $v(y) > 0$.
  If $v(x) < 0$ then $y = x \cdot (x^{-1} - 1)$ while $v(x^{-1}) > 0$.
  Since $f(-1) = 0$ and $f \in \Dcal_{v}$, it follows that
  \[ f(y) = f(x) + f(x^{-1}-1) = f(x) + f(1-x^{-1}) = f(x). \]
  We similarly have $g(y) = g(x)$, so the equation again holds true.
  The equation similarly holds true if $v(y) < 0$.

  The last case to consider is where $v(x) = v(y) = 0$, in which case $x,y \in \Urm_{v}$ and the values of $f$ and $g$ at $x$ and $y$ can be computed in the residue field.
  In other words, letting $\bar x$ and $\bar y$ denote the images of $x$ and $y$ in $Kv^\times$, we have $ \bar x + \bar y = 1$ so that
  \[ f(x) \cdot g(y) = f_{v}(\bar x) \cdot g_{v}(\bar y) = f_{v}(\bar y) \cdot g_{v}(\bar x) = f(y) \cdot g(x). \]
  In any case, we see that the necessary equation does indeed hold.
\end{proof}

\subsection{Valuative subspaces}

A closed subspace $\Ical$ of $\Gcal_{K}$ will be called \emph{valuative} provided that $\Ical \subset \Ical_{v}$ for some valuation $v$ of $K$.
\begin{lemma}\label{lemma:main-valuative}
  Suppose that $\Ical$ is valuative.
  Then there exists a unique minimal valuation $v_{\Ical}$ such that $\Ical \subset \Ical_{v_{\Ical}}$.
  The valuation $v = v_{\Ical}$ is characterized by the following two properties:
  \begin{enumerate}
    \item One has $\Ical \subset \Ical_{v}$, or equivalently $\Urm_{v} \subset \Ical^{\perp}$.
    \item The subgroup $v(\Ical^{\perp})$ contains no nontrivial convex subgroups.
  \end{enumerate}
\end{lemma}
\begin{proof}
  The collection of all valuations $w$ such that $\Ical \subset \Ical_{w}$ is nonempty by assumption.
  Also, if $w_{i}$ is a chain of such valuations, then the infimum $w$ of the $w_{i}$ satisfies
  \[ \Urm_{w} = \bigcup_{i} \Urm_{w_{i}} \]
  and as $\Urm_{w_{i}} \subset \Ical^{\perp}$ for all $i$, it follows that $\Urm_{w} \subset \Ical^{\perp}$ hence $w$ is also in the collection.
  This collection is also closed under binary infimums: If $w_{1}$ and $w_{2}$ are two such valuations and $w$ is their infimum, then
  \[ \Urm_{w_{1}} \cdot \Urm_{w_{2}} = \Urm_{w}, \]
  and since $\Urm_{w_{i}} \subset \Ical^{\perp}$, we also have $\Urm_{w} \subset \Ical^{\perp}$, hence $w$ is again in the collection.
  It follows that this collection has a unique minimal element $v_{\Ical}$.

  Put $v = v_{\Ical}$.
  Note that $v(\Ical^{\perp})$ contains no nontrivial convex subgroups for otherwise the coarsening associated to such a subgroup would contradict the minimality of $v$.
  Conversely, if $v$ satisfies (1) and (2) and $w$ is a coarsening of $v$ satisfying (1), then the convex subgroup of $vK$ associated to $w$ must be contained in $v(\Ical^{\perp})$.
  This subgroup must be trivial by condition (2) and thus $w = v$.
  It follows that $v$ is minimal with respect to condition (1), hence $v = v_{\Ical}$.
  This concludes the proof.
\end{proof}

\begin{lemma}\label{lemma:v_perp_saturation}
  Suppose that $v$ is a valuation of $K$ and that $\Hcal$ is a closed subspace of $\Gcal_{K}$.
  Then $v((\Ical_{v} \cap \Hcal)^{\perp})$ is the saturation of $v(\Hcal^{\perp})$ in $vK$.
  In particular, if $\Hcal \subset \Ical_{v}$ then $v(\Hcal^{\perp})$ is saturated.
\end{lemma}
\begin{proof}
  Put $H := (\Ical_{v} \cap \Hcal)^{\perp}$.
  First, note that $v(H)$ is indeed saturated in $vK$.
  Indeed, if $n \cdot v(t) \in v(H)$ for some $t \in K^{\times}$ and some positive integer $n$, then $t^{n} \in \Urm_{v} \cdot H$ while $\Urm_{v} \subset H$ so that $t^{n} \in H$.
  Since $H$ is itself saturated, it follows that $t \in H$ hence $v(t) \in v(H)$.

  Put $T := \Urm_{v} \cdot \Hcal^{\perp}$.
  Since $\Ical_{v}^{\perp} = \Urm_{v}$, it follows that $\Gcal_{K|T} = \Ical_{v} \cap \Hcal$ and thus $H$ is the saturation of $T$.
  This means that $H/T$ is torsion, hence $v(H)/v(T)$ is torsion as well, while $v(T) = v(\Hcal^{\perp})$.
  It follows that $v(H)$ is indeed the saturation of $v(\Hcal^{\perp})$.
\end{proof}

\begin{lemma}\label{lemma:valuative_iff_saturation}
  Let $\Hcal$ be a closed subspace of $\Gcal_{K}$, and let $v$ be a valuation of $K$.
  Then $v = v_{\Ical}$ for $\Ical = \Ical_{v} \cap \Hcal$ if and only if the saturation of $v(\Hcal^{\perp})$ in $vK$ contains no nontrivial convex subgroups.
\end{lemma}
\begin{proof}
  Combine Lemmas~\ref{lemma:main-valuative} and~\ref{lemma:v_perp_saturation}.
\end{proof}

\subsection{Detecting valuative subspaces}

\begin{lemma}\label{lemma:centralizer_valuative}
  Let $\Ical$ be a valuative subspace of $\Gcal_{K}$ with associated valuation $v := v_{\Ical}$.
  Then one has $\Dcal_{v} = \Ccal_{K}(\Ical)$.
\end{lemma}
\begin{proof}
  The inclusion $\Dcal_{v} \subset \Ccal_{K}(\Ical)$ follows from Lemma~\ref{lemma:Rcal_residue_compat}.
  Conversely, suppose that $f \in \Gcal_{K}$ satisfies $\Rcal_{K}(f,g)$ for all $g \in \Ical$.
  Let $x \in K^{\times}$ be an element satisfying $v(x) > 0$.

  Note that for all $g \in \Ical$, one has $g(1-x) = 0$ since $1-x \in \Urm_{v} \subset \Ical^{\perp}$.
  If there exists some $g \in \Ical$ such that $g(x) \neq 0$, then one has $f(1-x) = 0$ since
  \[ f(1-x) \cdot g(x) = f(x) \cdot g(1-x) = 0. \]

  Otherwise, there must exists some $y \in K^{\times}$ such that $0 < v(y) < v(x)$ and some $g \in \Ical$ such that $g(y) \neq 0$.
  Indeed, if this does not hold then $[0,v(x)] \subset vK$ would be contained in $v(\Ical^{\perp})$, which cannot happen since $v = v_{\Ical}$.
  With such a $y$, the argument above shows that $f(1-y) = 0$ while
  \[ f(1-x) = f((1-x) \cdot (1-y)) = f(1-(y + x \cdot (1-y))). \]
  But $v(y + x \cdot (1-y)) = v(y)$, so, again, the argument above shows that $f(1-x) = 0$.
  In other words, $\Urm_{v}^{1} \subset f^{\perp}$, hence $f \in \Dcal_{v}$.
\end{proof}

\begin{proposition}\label{proposition:centre_valuative}
  Suppose that $\Dcal$ is a closed subspace of $\Gcal_{K}$ such that $\Zcal_{K}(\Dcal) \neq \Dcal$.
  Then $\Ical := \Zcal_{K}(\Dcal)$ is valuative and $\Dcal \subset \Dcal_{v}$ for $v = v_{\Ical}$.
\end{proposition}
\begin{proof}
  By Lemma~\ref{lemma:centralizer_valuative} it suffices to show that $\Ical$ is valuative.
  Since $\Ical \neq \Dcal$, there exists $f_{1},f_{2} \in \Dcal$ such that $\Rcal_{K}(f_{1},f_{2})$ does not hold.
  Put $\Dcal_{i} = \Ical + \Qbb \cdot f_{i}$.
  Then $\Dcal_{i}$ are both closed subspaces of $\Gcal_{K}$, and the following hold:
  \begin{enumerate}
    \item $\Ical$ has codimension $1$ in $\Dcal_{i}$.
    \item $\Ical$ has codimension $2$ in $\Dcal_{1} + \Dcal_{2}$.
    \item $\Ical = \Dcal_{1} \cap \Dcal_{2}$.
    \item For all $f,g \in \Dcal_{i}$, one has $\Rcal_{K}(f,g)$.
  \end{enumerate}
  By Theorem~\ref{theorem:main_acl}, there exist valuations $v_{i}$ and closed subspaces $\Ical_{i}$ of $\Dcal_{i}$ of codimension $\leq 1$ such that $\Dcal_{i} \subset \Dcal_{v_{i}}$ and $\Ical_{i} \subset \Ical_{v_{i}}$.

  If $v_{1}$ and $v_{2}$ are \emph{not} comparable, then, letting $v$ denote their infimum, one has $\Urm_{v_{1}}^{1} \cdot \Urm_{v_{2}}^{1} = \Urm_{v}$ by the approximation theorem for independent valuations.
  It follows that $\Dcal_{v_{1}} \cap \Dcal_{v_{2}} = \Ical_{v}$, hence $\Ical \subset \Ical_{v}$, thereby concluding the proof.

  So, assume without loss of generality that $v_{1} \le v_{2}$ hence
  \[ \Ical_{v_{1}} \subset \Ical_{v_{2}} \subset \Dcal_{v_{2}} \subset \Dcal_{v_{1}}. \]
  Assume for a contradiction that $\Ical$ is \emph{not valuative}.
  Then $\Ical$ is not contained in $\Ical_{v_{i}} \cap \Dcal_{i}$, so we may find $g_{i} \in \Ical_{v_{i}} \cap \Dcal_{i}$ such that $\Dcal_{i} = \Ical + \Qbb \cdot g_{i}$.
  In particular, we have
  \[ \Dcal_{1} + \Dcal_{2} = \Ical + \Qbb \cdot g_{1} + \Qbb \cdot g_{2} \]
  while $g_{1},g_{2} \in \Ical_{v_{2}}$.
  It follows that $\Dcal_{1} + \Dcal_{2} \subset \Dcal_{v_{2}}$ and that the image of $\Dcal_{1} + \Dcal_{2}$ in $\Dcal_{v_{2}}/\Ical_{v_{2}}$ agrees with the image of $\Dcal_{2}$ in this quotient, which has dimension $\leq 1$.
  This together with Lemma~\ref{lemma:Rcal_residue_compat} would imply that $\Rcal_{K}(f,g)$ holds for all $f,g \in \Dcal_{1} + \Dcal_{2}$, which is impossible since $f_{1},f_{2} \in \Dcal_{1} + \Dcal_{2}$.
\end{proof}

\subsection{Visible valuations}

Let $T$ be a subgroup of $K^{\times}$ and $v$ a valuation of $K$.
We shall say that $v$ is \emph{$T$-visible} provided that the following hold:
\begin{enumerate}
  \item One has $\Ical_{v} \cap \Gcal_{K|T} = \Zcal_{K}(\Dcal_{v} \cap \Gcal_{K|T}) \neq \Dcal_{v} \cap \Gcal_{K|T}$.
  \item One has $v = v_{\Ical}$ for $\Ical = \Ical_{v} \cap \Gcal_{K|T}$.
\end{enumerate}
These are precisely the valuations $v$ for which we will be able to characterize $\Ical_{v} \cap \Gcal_{K|T}$ and $\Dcal_{v} \cap \Gcal_{K|T}$ using the relation $\Rcal_{K}$ restricted to $\Gcal_{K|T}$, as we show in the following theorem.
If $H$ is the saturation of $T$, then $\Gcal_{K|T} = \Gcal_{K|H}$ hence a valuation is $T$-visible if and only if it is $H$-visible.
When $T = k^{\times}$ for a subfield $k$ of $K$, then we shall say ``$k$-visible'' instead of ``$k^{\times}$-visible.''
When $T = \{1\}$ is trivial, we shall say ``visible'' instead of ``$\{1\}$-visible.''

\begin{theorem}\label{theorem:detect_visible_valuations}
  Let $T$ be a subgroup of $K^{\times}$ and $\Dcal \subset \Gcal_{K|T}$ be a closed subspace.
  There exists a $T$-visible valuation of $K$ such that $\Zcal_{K}(\Dcal) = \Ical_{v} \cap \Gcal_{K|T}$ and $\Dcal = \Dcal_{v} \cap \Gcal_{K|T}$ if and only if the following conditions hold:
  \begin{enumerate}
    \item One has $\Zcal_{K}(\Dcal) \neq \Dcal$.
    \item One has $\Ccal_{K}(\Zcal_{K}(\Dcal)) \cap \Gcal_{K|T} = \Dcal$.
  \end{enumerate}
\end{theorem}
\begin{proof}
  First suppose that $v$ is indeed $T$-visible.
  Put $\Dcal := \Dcal_{v} \cap \Gcal_{K|T}$ and $\Ical := \Ical_{v} \cap \Gcal_{K|T}$.
  By assumption, we have $\Ical = \Zcal_{K}(\Dcal) \neq \Dcal$ and by Lemma~\ref{lemma:centralizer_valuative} we have $\Ccal_{K}(\Ical) = \Dcal_{v}$ since $v = v_{\Ical}$, hence both conditions (1) and (2) hold true.

  Conversely, suppose that $\Dcal$ satisfies conditions (1) and (2) and put $\Ical := \Zcal_{K}(\Dcal)$.
  By Proposition~\ref{proposition:centre_valuative} and condition (1), $\Ical$ is valuative and, setting $v = v_{\Ical}$, one has $\Dcal \subset \Dcal_{v}$.
  Lemma~\ref{lemma:centralizer_valuative} shows that $\Ccal_{K}(\Ical) = \Dcal_{v}$ hence condition (2) implies that $\Dcal = \Dcal_{v} \cap \Gcal_{K|T}$.
  We also know that $\Ical_{v} \cap \Gcal_{K|T} \subset \Zcal_{K}(\Dcal_{v} \cap \Gcal_{K|T}) = \Ical$ by Lemma~\ref{lemma:Rcal_residue_compat}, while $\Ical \subset \Ical_{v} \cap \Gcal_{K|T}$ because $v = v_{\Ical}$.
  Thus $\Ical = \Ical_{v} \cap \Gcal_{K|T}$.
  The fact that $v$ is $T$-visible follows directly from the definition and the observations above.
\end{proof}

\subsection{Abundant visibility}

This section shows that fields of higher transcendence degree have an abundance of visible valuations.
\begin{lemma}\label{lemma:pos_trdeg_non_henselian}
  Suppose that $k$ is a subfield of $K$, and let $T$ be any subgroup of $K^{\times}$ which is contained in $k^{\times}$.
  Assume that $\trdeg(K|k) \geq 1$, and that $v$ is a valuation of $K$ such that $\Gcal_{K|T} \subset \Dcal_{v}$.
  Then $v$ is trivial.
\end{lemma}
\begin{proof}
  If not, then there exists some $t \in K$ which is transcendental over $k$ such that $v(t) > 0$.
  Thus $1 + t \in \Urm_{v}^{1} \subset \Dcal_{v}^{\perp} \subset \Gcal_{K|T}^{\perp} \subset \Gcal_{K|k}^{\perp}$.
  But $\Gcal_{K|k}^{\perp}$ is the saturation of $k^{\times}$ in $K^{\times}$.
  This implies that $1+t$ is algebraic over $k$, which is impossible.
\end{proof}

\begin{proposition}\label{proposition:visible_of_trdeg_ge_one}
  Suppose that $k$ is a subfield of $K$, and let $T$ be any subgroup of $K^{\times}$ which is contained in $k^{\times}$.
  Let $v$ be a valuation of $K$ such that the saturation of $vk$ in $vK$ contains no nontrivial convex subgroups and such that $\trdeg(Kv|kv) \geq 1$.
  Then $v$ is visible over $T$.
\end{proposition}
\begin{proof}
  Since the saturation of $T$ in $K^{\times}$ is contained in the radical closure of $k$ in $K$, we may assume without loss of generality that $T$ is saturated and that $k$ is radically closed in $K$.
  In particular, $k^{\times}$ is also saturated in $K^{\times}$ and hence $k^{\times} = \Gcal_{K|k}^{\perp}$ while $T = \Gcal_{K|T}^{\perp}$.
  By Lemma~\ref{lemma:v_perp_saturation}, we see that the saturation of $vk$ in $vK$ is $v(\Ical_{0}^{\perp})$ where $\Ical_{0} = \Ical_{v} \cap \Gcal_{K|k}$.
  On the other hand, our assumption ensures that $\Gcal_{K|k} \subset \Gcal_{K|T}$ hence $\Ical_{0} \subset \Ical_{1} := \Ical_{v} \cap \Gcal_{K|T}$, thus $\Ical_{1}^{\perp} \subset \Ical_{0}^{\perp}$.
  Thus $v(\Ical_{1}^{\perp})$ contains no nontrivial convex subgroups, so that $v$ is indeed the valuation associated to $\Ical_{1}$ due to the characterization from Lemma~\ref{lemma:main-valuative}.

  We must show that
  \[ \Ical_{v} \cap \Gcal_{K|T} = \Zcal_{K}(\Dcal_{v} \cap \Gcal_{K|T}) \neq \Dcal_{v} \cap \Gcal_{K|T}. \]
  Recall that $\Dcal_{v} \cap \Gcal_{K|T} / \Ical_{v} \cap \Gcal_{K|T} \cong \Gcal_{Kv|Tv}$.
  Since $Tv \subset kv^{\times}$ we find that $\Gcal_{Kv|kv} \subset \Gcal_{Kv|Tv}$, while $\trdeg(Kv|kv) \geq 1$ ensures that $\Gcal_{Kv|kv}$ is infinite dimensional.
  Thus $\Ical_{v} \cap \Gcal_{K|T} \neq \Dcal_{v} \cap \Gcal_{K|T}$.

  Thus all that remains to show is that $\Ical_{v} \cap \Gcal_{K|T} = \Zcal_{K}(\Dcal_{v} \cap \Gcal_{K|T})$.
  Note that $\Ical_{v} \cap \Gcal_{K|T} \subset \Zcal_{K}(\Dcal_{v} \cap \Gcal_{K|T})$ by Lemma~\ref{lemma:Rcal_residue_compat}, so we only need to show the other inclusion.
  As noted above, the image of the composition
  \[ \Dcal_{v} \cap \Gcal_{K|T} \hookrightarrow \Dcal_{v} \twoheadrightarrow \Dcal_{v}/\Ical_{v} \cong \Gcal_{Kv} \]
  is precisely $\Gcal_{Kv|Tv}$.
  By Lemma~\ref{lemma:Rcal_residue_compat}, it follows that the image of $\Zcal_{K}(\Dcal_{v} \cap \Gcal_{K|T})$ in $\Gcal_{Kv|Tv}$ is precisely $\Zcal_{Kv}(\Gcal_{Kv|Tv})$.
  It suffices to show that this image, or equivalently $\Zcal_{Kv}(\Gcal_{Kv|Tv})$, is trivial.

  We now have two cases to consider.
  If $\Zcal_{Kv}(\Gcal_{Kv|Tv}) = \Gcal_{Kv|Tv}$ then by Theorem~\ref{theorem:main_acl}, there exists a valuation $w$ of $Kv$ such that $\Gcal_{Kv|Tv} \subset \Dcal_{w}$ and such that $\Ical_{w} \cap \Gcal_{Kv|Tv}$ has codimension $\le 1$ in $\Gcal_{Kv|Tv}$.
  Lemma~\ref{lemma:pos_trdeg_non_henselian} shows that $w$ is trivial and thus $\Ical_{w}$ is trivial, which is impossible since $\Gcal_{Kv|kv} \subset \Gcal_{Kv|Tv}$ is infinite dimensional.

  Thus we have $\Zcal_{Kv}(\Gcal_{Kv|Tv}) \neq \Gcal_{Kv|Tv}$.
  In this case, Proposition~\ref{proposition:centre_valuative} shows that $\Zcal_{Kv}(\Gcal_{Kv|Tv})$ is valuative and, letting $w$ denote the associated valuation, one has $\Gcal_{Kv|Tv} \subset \Dcal_{w}$.
  Again, Lemma~\ref{lemma:pos_trdeg_non_henselian} shows that $w$ is trivial, hence $\Ical_{w}$ is trivial, so $\Zcal_{Kv}(\Gcal_{Kv|Tv})$ is trivial as well.
  This concludes the proof of the proposition.
\end{proof}

\section{Algebraic dependence}\label{section:algebraic_dependence}

Let $k$ be a relatively algebraically closed subfield of $K$.
Our goal in this section is to provide a characterization of algebraic dependence (over $k$) in $K$ using the algebra $\Kcal_{*}(K|k)$, or equivalently, using the structure $\Ascr(K|k)$.

\subsection{Milnor-closed subspaces}

Let $T$ be a subgroup of $K^{\times}$ and $\Hcal$ a subspace of $\Kcal_{K|T}$.
We say that $\Hcal$ is \emph{Milnor-closed} provided that for all nontrivial $s \in \Hcal$ and $t \in \Kcal_{K|T}$ such that $\{s,t\} = 0$ in $\Kcal_{2}(K|T)$, one has $t \in \Hcal$ as well.
Any subset $S$ of $\Kcal_{K|T}$ has a \emph{Milnor-closure} which is the smallest Milnor-closed subspace $\Hcal$ of $\Kcal_{K|T}$ that contains $S$.
Explicitly, the Milnor-closure of $S$ can be computed as a union
\[ \Hcal = \bigcup_{n = 0}^{\infty} \Hcal_{n} \]
where $\Hcal_{0}$ is the subspace generated by $S$, and $\Hcal_{n+1}$ is the subspace generated by $\Hcal_{n}$ and all $t \in \Kcal_{K|T}$ such that there exists some nontrivial $s \in \Hcal_{n}$ where $\{s,t\} = 0$ in $\Kcal_{2}(K|T)$.

\begin{lemma}\label{lemma:milnor_closure_description}
  Let $\Hcal$ be a Milnor-closed subspace of $\Kcal_{K|T}$.
  Let $H$ be the preimage of $\Hcal$ with respect to the canonical map $K^{\times} \to \Kcal_{K|T}$.
  Then $H$ contains the following:
  \begin{enumerate}
    \item The saturation of $T$.
    \item Elements of $K^{\times}$ of the form $a + b \cdot h$ for any $h \in H$ whose image in $\Kcal_{K|T}$ is nontrivial and any $a,b$ in the saturation of $T$.
  \end{enumerate}
\end{lemma}
\begin{proof}
  Let $R$ denote the saturation of $T$.
  It is clear that $H$ contains $R$.
  Recall that we have a natural morphism of graded rings
  \[ \KM_{*}(K|R) \to \Kcal_{*}(K|R) = \Kcal_{*}(K|T). \]
  Consider elements of $K^{\times}$ the form $a + b \cdot h$ as described in (3).
  We calculate some symbols in $\KM_{2}(K|R)$:
  \begin{align*}
    \{h, a + b \cdot h\} &= \{h, 1 - (-b \cdot a^{-1}) \cdot h \} \\
                         &= \{(-b \cdot a^{-1}) \cdot h, 1 - (-b \cdot a^{-1}) \cdot h\} \\
                         &= 0.
  \end{align*}
  The first equality follows from the fact that $a^{-1} \in R$, the second from the fact that $-b \cdot a^{-1} \in R$, and the last from the Steinberg relations in Milnor K-theory.
  Since $\Hcal$ is Milnor-closed and $h \in H$ it follows that $a + b \cdot h \in H$ as well.
\end{proof}

\subsection{Detecting algebraic dependence}

\begin{lemma}\label{lemma:mod_units_exterior}
  Let $v$ be a valuation on $K$.
  The canonical map
  \[ \wedge^{*} (\Kcal_{K|\Urm_{v}}) \to \Kcal_{*}(K|\Urm_{v}) \]
  given by the identity in degree one is an isomorphism.
\end{lemma}
\begin{proof}
  Note that the map is surjective and that $\Urm_{v} = \Ical_{v}^{\perp}$.
  We will identify $\Kcal_{K|\Urm_{v}}$ with $\Qbb \otimes_{\Zbb} vK$ and for $t \in K^{\times}$, we abuse the notation and write $v(t)$ for the image of $t$ in $\Qbb \otimes_{\Zbb} vK = \Kcal_{K|\Urm_{v}}$.
  The kernel of the map in question is generated by $r_{x,y} := v(x) \wedge v(y)$ where $x,y \in K^{\times}$ satisfy $x + y = 1$.
  We claim that all such $r_{x,y}$ are already trivial.
  If $v(x) = 0$ or $v(y) = 0$ then $r_{x,y} = 0$, so there is nothing to show.
  If $v(x) > 0$ then $v(y) = 0$ so $r_{x,y} = 0$, and similarly if $v(y) > 0$.
  Otherwise, $v(x) < 0$ and $v(y) = v(x)$ so again $r_{x,y} = 0$.
  In any case, the kernel is trivial, so the map in question is an isomorphism.
\end{proof}

\begin{lemma}\label{lemma:detect_independence}
  Suppose that $k$ is a subfield of $K$ and $T$ is a subgroup of $K^{\times}$ which is contained in $k^{\times}$.
  Let $t_{1},\ldots,t_{n} \in K^{\times}$ be given, and let $\bar t_{i}$ denote the image of $t_{i}$ in $\Kcal_{K|T}$.
  If $\{\bar t_{1},\ldots,\bar t_{n}\} = 0$ in $\Kcal_{n}(K|T)$, then $t_{1},\ldots,t_{n}$ are algebraically dependent over $k$.
\end{lemma}
\begin{proof}
  If $t_{1},\ldots,t_{n}$ are algebraically independent, then we may find a $k$-valuation $v$ on $K$ such that $v(t_{1}),\ldots,v(t_{n})$ are linearly independent in $\Qbb \otimes_{\Zbb} vK = \Kcal_{K|\Urm_{v}}$.
  For example, we can take the discrete rank $n$ valuation on $k(t_{1},\ldots,t_{n})$ associated to the regular sequence $(t_{1},\ldots,t_{n})$ and choose $v$ to be some prolongation of this valuation to $K$.
  Since $T \subset k^{\times} \subset \Urm_{v}$, the assertion follows from Lemma~\ref{lemma:mod_units_exterior}.
\end{proof}

\subsection{Geometric lattices}

Suppose that $k$ is a relatively algebraically closed subfield of $K$.
The collection of all relatively algebraically closed subextension of $K|k$ will be denoted by $\Gbb(K|k)$.
This is a \emph{complete} lattice, with respect to inclusion of subfields of $K$, meaning that any set has a greatest lower bound (the infimum) and a smallest upper bound (the supremum).
In this case, the infimum is computed by taking intersections in $K$, and the supremum is computed by taking the relative algebraic closure of the compositum in $K$.
We call $\Gbb(K|k)$ the \emph{geometric lattice} associated to $K|k$.

For a field $F$, write $F^{i}$ for the perfect closure of $F$.
Note that restriction along $K \hookrightarrow K^{i}$ induces an isomorphism of geometric lattices
\[ \Gbb(K^{i}|k^{i}) \cong \Gbb(K|k), \]
where the inverse is given by $M \mapsto M^{i}$.
We will make use of the following result from~\cite{zbMATH05350217}, which fundamentally relies on the work of Evans-Hrushovski~\cite{zbMATH00007333,zbMATH00839199} based on the \emph{group configuration theorem}.
This theorem should be thought of as an analogue of the fundamental theorem of projective geometry, but for an incidence geometry associated to $\Gbb(K|k)$ as opposed to a projective space.

\begin{theorem}[\cite{zbMATH05350217}, Theorem 4.2]\label{theorem:Evans-Hrushovski-Gismatullin}
  Suppose that $K|k$ and $L|l$ are relatively algebraically closed extensions of fields.
  Assume that $\trdeg(K|k) \geq 5$, and that $\phi : \Gbb(K|k) \cong \Gbb(L|l)$ is an isomorphism of geometric lattices.
  Then there exists an isomorphism $\Phi : K^{i} \cong L^{i}$ of fields satisfying $\Phi(k^{i}) = l^{i}$ such that $\phi(M) = \Phi(M^{i}) \cap L$ for all $M \in \Gbb(K|k)$.
  Furthermore, $\Phi$ is unique with these properties up-to composition with some power of the $p$-power Frobenius $x \mapsto x^{p}$, where $p$ is the characteristic exponent of $K$.
\end{theorem}
\begin{proof}
  Since $\Gbb(K|k) \cong \Gbb(L|l)$, we have $\trdeg(L|l) = \trdeg(K|k) \geq 5$, as the transcendence degree of a relatively algebraically closed field extension is the Krull dimension of the associated geometric lattice.

  The only part that doesn't follow immediately from~\cite[Theorem 4.2]{zbMATH05350217} is that~\cite{zbMATH00007333,zbMATH00839199,zbMATH05350217} all write $\Gbb(K|k)$ for the \emph{combinatorial geometry} associated to $K|k$ as opposed to the geometric lattice as we have defined above.
  But the two approaches are easily seen to be equivalent (this is a very well-known fact of matroid theory).

  Indeed, let us write $\Gbb'(K|k)$ for the combinatorial geometry associated to $K|k$.
  This object refers to the set of all relatively algebraically closed subextensions $M$ of $K|k$ such that $\trdeg(M|k) = 1$, and $\Gbb'(K|k)$ is endowed with a \emph{closure operator} $\cl$ which associates to a subset $S \subset \Gbb'(K|k)$ the set $\cl(S)$ of all $M \in \Gbb'(K|k)$ such that $M \subset \bar{k(S)} \cap K$.
  A subset $S$ of $\Gbb'(K|k)$ is called \emph{closed} provided that $\cl(S) = S$.

  One can functorially recover $\Gbb'(K|k)$ from $\Gbb(K|k)$, and vice-versa, as follows.
  First note that $\Gbb(K|k)$ has a unique minimal element $\bot$ corresponding to the field $k$.
  Next, note that $\Gbb'(K|k)$ is the set of \emph{atoms} of $\Gbb(K|k)$, i.e.~the set of elements $M$ of $\Gbb(K|k)$ which are different from $\bot$ and minimal with that property.
  The closure operator $\cl$ is obtained using the lattice structure of $\Gbb(K|k)$ as follows:
  \[ \cl(S) = \{M \in \Gbb'(K|k) \ | \ M \le \Sup(S) \}, \]
  where $\Sup(S)$ denotes the supremum of $S$ in $\Gbb(K|k)$.
  Conversely, one may identify $\Gbb(K|k)$ with the lattice of closed subsets of $\Gbb'(K|k)$.

  Going back to the context of the theorem, we have an isomorphism $\phi : \Gbb(K|k) \cong \Gbb(L|l)$ of lattices, which induces an isomorphism $\phi' : \Gbb'(K|k) \cong \Gbb'(L|l)$ of combinatorial geometries.
  By~\cite[Theorem 4.2(ii)]{zbMATH05350217}, there exists an isomorphism $\Phi : K^{i} \cong L^{i}$ of fields with $\Phi(k^{i}) = l^{i}$ such that for all $M \in \Gbb'(K|k)$, one has $\phi'(M) = \Phi(M^{i}) \cap L$.

  It remains to show that $\phi(M) = \Phi(M^{i}) \cap L$ for \emph{all} $M \in \Gbb(K|k)$, but this follows easily from the fact that $\Gbb(K|k)$ and $\Gbb(L|l)$ are \emph{atomistic} lattices, meaning that every element $M$ of $\Gbb(K|k)$ is the supremum of the atoms it bounds from above (and similarly for $\Gbb(L|l)$).
\end{proof}

\subsection{Geometric subspaces}

Suppose that $k$ is a relatively algebraically closed subfield of $K$.
Let $L$ be any subextension of $K|k$.
Then the canonical map
\[ \Kcal_{L|k} \to \Kcal_{K|k} \]
is injective, and we will identify $\Kcal_{L|k}$ with its image in $\Kcal_{K|k}$.
A subspace of $\Kcal_{K|k}$ will be called \emph{geometric} if it is of the form $\Kcal_{L|k}$ for some \emph{relatively algebraically closed} subextension $L$ of $K|k$.
The collection of all geometric subspaces of $\Kcal_{K|k}$ will be denoted by $\Gbb_{\Kcal}(K|k)$, considered as a poset with respect to inclusion in $\Kcal_{K|k}$.

\begin{proposition}\label{proposition:geometric_lattice_iso}
  The canonical map $\Gbb(K|k) \to \Gbb_{\Kcal}(K|k)$ sending $L \in \Gbb(K|k)$ to $\Kcal_{L|k}$ is an order isomorphism $\Gbb(K|k) \cong \Gbb_{\Kcal}(K|k)$.
  In particular, $\Gbb_{\Kcal}(K|k)$ is also a complete lattice.
\end{proposition}
\begin{proof}
  This map is clearly monotone and surjective.
  Conversely, suppose that $L_{1}$ and $L_{2}$ are two elements of $\Gbb(K|k)$, that $\Kcal_{L_{1}|k} \le \Kcal_{L_{2}|k}$ and that $t \in L_{1}$ is some element.
  Letting $\bar t$ denote the image of $t$ in $\Kcal_{K|k}$, we have $\bar t \in \Kcal_{L_{1}|k}$, which is contained in $\Kcal_{L_{2}|k}$.
  Thus there exists some positive integer $n$ and some constant $c \in k^{\times}$ such that $c \cdot t^{n} \in L_{2}^{\times}$.
  This implies that $t$ is algebraic over $L_{2}$ and thus $t \in L_{2}$.
  In other words, we have $L_{1} \le L_{2}$ in $\Gbb(K|k)$.
  In particular, the map in question is then injective, hence bijective, and its inverse is also monotone.
\end{proof}

\subsection{Milnor closure vs.~algebraic closure}

We continue to work with a relatively algebraically closed subfield $k$ of $K$.
Recall that we have a canonical pairing
\[ \Kcal_{K|k} \times \Gcal_{K|k} \to \Qbb. \]
For the rest of this subsection, we will use the notation $\Ecal^{\perp}$ to denote the orthogonal of a subspace $\Ecal \subset \Gcal_{K|k}$ with respect to the above pairing.
Although this overloads the notation $(-)^{\perp}$ introduced previously, it should be clear from context when $\Ecal^{\perp}$ refers to a subspace of $\Kcal_{K|k}$ as opposed to a subgroup of $K^{\times}$.

The following theorem is the technical core of this paper.
It provides a characterization of the elements $\Gbb_{\Kcal}(K|k)$ as subspaces of $\Kcal_{K|k}$.

\begin{theorem}\label{theorem:recover_Kcal_geom_lattice}
  Assume that $\trdeg(K|k) \geq 2$.
  Let $k$ be a relatively algebraically closed subfield of $K$, and let $\Hcal$ be a Milnor-closed subspace of $\Kcal_{K|k}$.
  Let $H$ denote the preimage of $\Hcal$ with respect to the map $K^{\times} \to \Kcal_{K|k}$, and let $L$ denote the relative algebraic closure of $k(H)$ in $K$.
  Let $\Vscr_{\Hcal}$ denote the collection of all closed subspaces $\Dcal$ of $\Gcal_{K|k}$ satisfying the following conditions:
  \begin{enumerate}
    \item One has $\Zcal_{K}(\Dcal) \neq \Dcal$.
    \item One has $\Ccal_{K}(\Zcal_{K}(\Dcal)) \cap \Gcal_{K|k} = \Dcal$.
    \item One has $\Dcal^{\perp} \cap \Hcal = 0$.
  \end{enumerate}
  Then one has
  \[ \Kcal_{L|k} = \bigcap_{\Dcal \in \Vscr_{\Hcal}} \Zcal_{K}(\Dcal)^{\perp}. \]
  Conversely, any geometric subspace of $\Kcal_{K|k}$ arises in this way.
  More precisely, if $\Hcal = \Kcal_{L|k}$ for $L \in \Gbb(K|k)$, then $\Hcal$ is Milnor-closed and one has
  \[ \Hcal = \bigcap_{\Dcal \in \Vscr_{\Hcal}} \Zcal_{K}(\Dcal)^{\perp}. \]
\end{theorem}
\begin{proof}
  Let $\Vscr_{\Hcal}$ be as in the statement of the theorem.
  By Theorem~\ref{theorem:detect_visible_valuations}, the subspaces $\Dcal \in \Vscr_{\Hcal}$ all have the form $\Dcal = \Dcal_{v} \cap \Gcal_{K|k}$ where $v$ is a $k$-visible valuation of $K$ such that $\Ucal_{v|k}^{1} \cap \Hcal = 0$.
  Furthermore, in this case one has $\Zcal_{K}(\Dcal) = \Ical_{v} \cap \Gcal_{K|k}$ hence also $\Zcal_{K}(\Dcal)^{\perp} = \Ucal_{v|k}$.
  Thus, the intersection in question is precisely 
  \[ \Delta := \bigcap_{v} \Ucal_{v|k} \]
  where $v$ varies over the $k$-visible valuations of $K$ such that $\Ucal_{v|k}^{1} \cap \Hcal = 0$.

  Let us first handle the case where $\Hcal = 0$, so $L = k$.
  By the above discussion, it suffices to show that for every nontrivial $x \in \Kcal_{K|k}$, there exists some $k$-visible valuation $v$ of $K$ such that $x \notin \Ucal_{v|k}$.
  Replace $x$ by $n \cdot x$ for some positive integer $n$ to assume without loss of generality that $x$ is the image of a transcendental element $t \in K^{\times}$.
  Extend $t$ to a transcendence base
  \[ \Bcal = \{t\} \cup \Bcal_{0} \]
  of $K|k$, write $M = k(\Bcal_{0})$, and let $v$ be a prolongation of the $t$-adic valuation on $M(t)$ to $K$.
  Mote that $Kv|Mv$ is algebraic and $Mv = M$ while $vk = 0$.
  By Proposition~\ref{proposition:visible_of_trdeg_ge_one} and our assumptions on $\trdeg(K|k)$, we see that $v$ is indeed visible, while $v(t) > 0$ by construction.
  Finally, if $x \in \Ucal_{v|k}$ then $c \cdot t \in \Urm_{v}$ for some $c \in k^{\times}$, but $v(c \cdot t) = v(t) > 0$, so this cannot happen.
  Thus the assertion of the theorem holds true in the case where $\Hcal = 0$.
  Assume for the rest of the proof that $\Hcal \neq 0$.

  In order to show that $\Delta = \Kcal_{L|k}$, it suffices to show that a $k$-visible valuation $v$ of $K$ is trivial on $L$ if and only if $\Ucal_{v|k}^{1} \cap \Hcal = 0$.
  Indeed, in this case $\Delta = \bigcap_{v} \Ucal_{v|k}$ where $v$ varies over the $k$-visible valuations which are trivial on $L$.
  Hence $\Kcal_{L|k} \subset \Delta$.
  If $x \in \Kcal_{K|k} \smin \Kcal_{L|k}$, then choose some $t \in K \smin L$ such that the image of $t$ in $\Kcal_{K|k}$ is $n \cdot x$ for some positive integer $n$.
  This $t$ is then transcendental over $L$.
  Complete $t$ to a transcendence basis $\Bcal = \{t\} \cup \Bcal_{0}$ of $K|L$, put $M := L(\Bcal_{0})$, and let $v$ be an extension of the $t$-adic valuation on $M(t)$ to $K$.
  Arguing similarly to the above, we see that $v$ is $k$-visible, trivial on $L$, and that $x \notin \Ucal_{v|k}$.
  This shows that indeed, $\Delta = \Kcal_{L|k}$ provided that a $k$-visible valuation is trivial on $L$ if and only if $\Ucal_{v|k}^{1} \cap \Hcal = 0$.

  Suppose that $v$ is trivial on $L$.
  We have a canonical injective map
  \[ L^{\times}/k^{\times} \to Kv^{\times}/kv^{\times} \]
  and after tensoring with $\Qbb$ we obtain
  \[ \Kcal_{L|k} \to \Kcal_{Kv|kv} \]
  which is again injective.
  This last map is precisely the composition
  \[ \Kcal_{L|k} \hookrightarrow \Ucal_{v|k} \twoheadrightarrow \Ucal_{v|k}/\Ucal_{v|k}^{1} = \Kcal_{Kv|kv} \]
  and thus $\Ucal_{v|k}^{1} \cap \Kcal_{L|k} = 0$.
  Since $\Hcal \subset \Kcal_{L|k}$, we also have $\Ucal_{v|k}^{1} \cap \Hcal = 0$.

  We must now show that a $k$-visible valuation $v$ is trivial on $L$ provided that $\Ucal_{v|k}^{1} \cap \Hcal = 0$.
  \emph{This is the crux of the proof.}
  So, assume that $v$ is $k$-visible and \emph{nontrivial} on $L$.
  Since $L|k(H)$ is algebraic, it follows that $v$ is nontrivial on $k(H)$.

  We observe that $k(H)^{\times} = H$ as subgroups of $K^{\times}$.
  Indeed, first note that $k^{\times} \subset H$ and $H$ is multiplicatively closed.
  Thus, it suffices to show that $H \cup \{0\}$ is additively closed, and since $H$ is a subgroup, for this it suffices to show that $1 + t \in H \cup \{0\}$ whenever $t \in H$.
  If $t \in k^{\times}$, then this is obvious, and if not, then its image in $\Kcal_{K|k}$ is nontrivial, so that $1 + t \in H$ by Lemma~\ref{lemma:milnor_closure_description}, using the assumption that $\Hcal$ is Milnor-closed.

  In any case, we have $k(H)^{\times} = H$.
  Since $\Hcal \neq 0$, hence $k^{\times} \neq k(H)^{\times} = H$ there must exist some element $t \in H \smin k^{\times}$ such that $v(t) > 0$.
  Since $1 + t \in H \smin k^{\times}$ as well, the image of $1 + t$ in $\Kcal_{K|k}$ is a nontrivial element of $\Ucal_{v|k}^{1} \cap \Hcal$, showing that $\Ucal_{v|k}^{1} \cap \Hcal \neq 0$.

  The final assertion is easy.
  If $\Hcal = \Kcal_{L|k}$ then $\Hcal$ is Milnor-closed by Lemma~\ref{lemma:detect_independence}, while the preimage of $\Kcal_{L|k}$ in $K^{\times}$ is $L^{\times}$ since $L$ is relatively algebraically closed in $K$.
  So, the first part of the theorem, which is already proved, gives the desired claim.
\end{proof}

\subsection{Recovering transcendence degree}

We continue with the context above where $k$ is a relatively algebraically closed subfield of $K$.
\begin{lemma}\label{lemma:trdeg_geq_of_exists_subspace}
  Let $d$ be any positive integer.
  Assume that there exists a closed subspace $\Dcal$ of $\Gcal_{K|k}$ such that the following conditions hold true:
  \begin{enumerate}
    \item One has $\Zcal_{K}(\Dcal) \neq \Dcal = \Ccal_{K}(\Zcal_{K}(\Dcal)) \cap \Gcal_{K|k}$.
    \item One has $d \le \dim_{K}(\Zcal_{K}(\Dcal))$.
  \end{enumerate}
  Then $\trdeg(K|k) \geq d$.
\end{lemma}
\begin{proof}
  This is a simple consequence of Theorem~\ref{theorem:detect_visible_valuations} in conjunction with Abhyankar's inequality.
  By Theorem~\ref{theorem:detect_visible_valuations}, there exists a $k$-visible valuation $v$ of $K$ such that $\Ical := \Zcal_{K}(\Dcal) = \Ical_{v} \cap \Gcal_{K|k}$.
  Dualizing $\Ical$, we obtain $\Kcal_{K|k}/\Ucal_{v|k} = \Qbb \otimes_{\Zbb} (vK/vk)$, and our assumption tells us that $\dim_{\Qbb}(\Qbb \otimes_{\Zbb}(vK/vk)) \geq d$.
  Choose $t_{1},\ldots,t_{d} \in K$ whose images in $\Qbb \otimes_{\Zbb}(vK/vk)$ are linearly independent, and consider $L := k(t_{1},\ldots,t_{d})$, a finitely-generated subextension of $K|k$, which has the property that
  \[ \dim_{\Qbb}(\Qbb \otimes_{\Zbb} (vL/vk)) \geq d. \]
  By Abhyankar's inequality, we find
  \[ d \le \dim_{\Qbb}(\Qbb \otimes_{\Zbb} (vL/vk)) \le \trdeg(L|k) \le \trdeg(K|k), \]
  as claimed in the lemma.
\end{proof}

\begin{lemma}\label{lemma:exists_subspace_of_trdeg_gt}
  Suppose that $d$ is a positive integer and that $\trdeg(K|k) > d$.
  Then there exists a closed subspace $\Dcal$ of $\Gcal_{K|k}$ such that the following conditions hold true:
  \begin{enumerate}
    \item One has $\Zcal_{K}(\Dcal) \neq \Dcal = \Ccal_{K}(\Zcal_{K}(\Dcal)) \cap \Gcal_{K|k}$.
    \item One has $d \le \dim_{K}(\Zcal_{K}(\Dcal))$.
  \end{enumerate}
\end{lemma}
\begin{proof}
  Let $t_{1},\ldots,t_{d} \in K$ be algebraically independent over $k$, and extend $\{t_{1},\ldots,t_{d}\}$ to a transcendence base of $K|k$ of the form
  \[ \Bcal = \{t_{1},\ldots,t_{d}\} \cup \Bcal_{0}. \]
  By assumption on $\trdeg(K|k)$, the set $\Bcal_{0}$ is nonempty.
  Put $L := k(\Bcal_{0})$ and $M := k(\Bcal) = L(t_{1},\ldots,t_{d})$.
  Consider the valuation associated to the system of regular parameters $(t_{1},\ldots,t_{d})$ on $M$. This is a discrete rank $d$ valuation which is trivial on $L$.
  Extend it in some way to a valuation $v$ on $K$.
  Since $\trdeg(L|k) > 0$, it follows from Proposition~\ref{proposition:visible_of_trdeg_ge_one} that $v$ is visible, while $vk = 0$ and the images of $t_{1},\ldots,t_{d}$ in $\Qbb \otimes_{\Zbb} vK$ are rationally independent.
  Thus $\dim_{\Qbb}(\Ical_{v} \cap \Gcal_{K|k}) \geq d$ as well, while condition (1) follows from Theorem~\ref{theorem:detect_visible_valuations}.
\end{proof}

We will use Lemmas~\ref{lemma:trdeg_geq_of_exists_subspace} and~\ref{lemma:exists_subspace_of_trdeg_gt} primarily to provide a \emph{lower bound} on $\trdeg(K|k)$, which will be required in order to apply Theorem~\ref{theorem:recover_Kcal_geom_lattice}.
We summarize this observation in the following lemmas.
\begin{lemma}\label{lemma:le_trdeg_of_iso_of_lt_trdeg}
  Suppose that $K|k$ and $L|l$ are two relatively algebraically closed field extensions, and that $\Kcal_{*}(K|k) \cong \Kcal_{*}(L|l)$ as algebras.
  If $d$ is a positive integer such that $d < \trdeg(K|k)$, then $d \leq \trdeg(L|l)$.
\end{lemma}
\begin{proof}
  The isomorphism $\Kcal_{*}(K|k) \cong \Kcal_{*}(L|l)$ induces an isomorphism of structures $\Ascr(K|k) \cong \Ascr(L|l)$ by Fact~\ref{fact:alternating_iff_milnor_condition} and the surrounding discussion.
  By Lemma~\ref{lemma:exists_subspace_of_trdeg_gt}, there exists a closed subspace $\Dcal$ of $\Gcal_{K|k}$ such that
  \begin{enumerate}
    \item One has $\Zcal_{K}(\Dcal) \neq \Dcal = \Ccal_{K}(\Dcal) \cap \Gcal_{K|k}$.
    \item One has $d \le \dim_{K}(\Zcal_{K}(\Dcal))$.
  \end{enumerate}
  Transferring this subspace across the isomorphism $\Gcal_{K|k} \cong \Gcal_{L|k}$, we obtain a closed subspace of $\Gcal_{L|l}$ satisfying the conditions of Lemma~\ref{lemma:trdeg_geq_of_exists_subspace}, which implies that $d \le \trdeg(L|l)$.
\end{proof}

\begin{lemma}\label{lemma:trdeg_eq_of_iso}
  Suppose that $K|k$ and $L|l$ are two relatively algebraically closed field extensions, and that $\Kcal_{*}(K|k) \cong \Kcal_{*}(L|l)$ as algebras.
  Assume that $\trdeg(K|k) \geq 3$.
  Then $\trdeg(K|k) = \trdeg(L|l)$.
\end{lemma}
\begin{proof}
  By Lemma~\ref{lemma:trdeg_geq_of_exists_subspace}, we see that $\trdeg(L|l) \geq 2$, hence Theorem~\ref{theorem:recover_Kcal_geom_lattice} applies to both $K|k$ and $L|l$.
  This theorem provides us with an isomorphism of lattices
  \[ \Gbb_{\Kcal}(K|k) \cong \Gbb_{\Kcal}(L|l) \]
  hence also $\Gbb(K|k) \cong \Gbb(L|l)$ by Proposition~\ref{proposition:geometric_lattice_iso}.
  Since the transcendence degree of a relatively algebraically closed field extension is the Krull dimension of the corresponding geometric lattice, the claim follows.
\end{proof}

\section{Main results}

In this section, we present and prove the main results of this paper.
We split up this section into three subsections:
\begin{enumerate}
  \item The first regarding \emph{relative results}, dealing with relatively algebraically closed field extensions $K|k$ of sufficiently large transcendence degree, while using $\Kcal_{*}(K|k)$.
  \item The second regarding \emph{absolute results}, dealing with arbitrary fields of sufficiently large \emph{absolute} transcendence degree (i.e.~transcendence degree over the prime subfield), while using $\Kcal_{*}(K)$.
  \item The third dealing with finitely-generated relatively algebraically closed extensions $K|k$ over perfect fields and finitely-generated fields $L$, using $\KM_{*}(K|k)$, $\KM_{*}(L)$.
\end{enumerate}
In cases (1) and (2), our key tool is one of the main results of Evans-Hrushovski and Gismatullin~\cite{zbMATH00007333,zbMATH00839199,zbMATH05350217}, which we have summarized in Theorem~\ref{theorem:Evans-Hrushovski-Gismatullin} above.
In case (3), our key tools will be the reconstruction result due to Cadoret-Pirutka~\cite[Theorem 4]{zbMATH07463742}.

In the first two cases, we will only be able to recover fields \emph{up-to inseparable extensions}, or more precisely, we will be able to recover the perfect closure of the field in question.
Recall that we write $F^{i}$ for its perfect closure of a field $F$.
Of course, if $F$ has characteristic zero, then $F = F^{i}$.
If $F$ has positive characteristic $p$, then one has
\[ F^{i} = \bigcup_{n \geq 0} F^{1/p^{n}}. \]

Note that if $K|k$ is a relatively algebraically closed extension of fields, then $K^{i}|k^{i}$ is also relatively algebraically closed.
The canonical map $K^{\times} \to (K^{i})^{\times}$ induces an \emph{injective} map $K^{\times}/k^{\times} \to (K^{i})^{\times}/(k^{i})^{\times}$, and the induced map of $\Qbb$-modules $\Kcal_{K|k} \to \Kcal_{K^{i}|k^{i}}$ is an \emph{isomorphism}.

We first prove an auxiliary lemma that will be necessary to show certain \emph{uniqueness} properties of the isomorphisms of fields that we obtain.
This lemma has appeared before in~\cite[Theorem 1.1]{zbMATH00839199} and in~\cite[Lemma 13]{zbMATH05635168}, although the proof of~\cite{zbMATH00839199} is more complicated as it relies on the more general result~\cite[Theorem 2.2.2]{zbMATH00007333}, while~\cite{zbMATH05635168} has a blanket assumption that the fields in question have characteristic zero and the base field is algebraically closed.
The argument we give is an adaptation of the proof from~\cite{zbMATH05635168} which avoids the restrictions on the base field.

\begin{lemma}\label{lemma:acl_mul_compat}
  Let $K|k$ be a relatively algebraically closed extension of fields.
  Suppose that $x,y \in K$ are algebraically independent over $k$.
  Let $a,b \in K \smin k$ be two elements such that $\trdeg(k(a,x)|k) = 1$, $\trdeg(k(b,y)|k) = 1$ and $\trdeg(k(a \cdot b, x \cdot y)|k) = 1$.
  Then there exist nonzero integers $m$ and $n$ such that, modulo $k^{\times}$, one has $a^{n} = x^{m}$, $b^{n} = y^{m}$.
\end{lemma}
\begin{proof}
  Embed $K^{\times}/k^{\times}$ into $\bar K^{\times}/\bar k^{\times}$ to assume without loss of generality that both $k$ and $K$ are algebraically closed.
  The elements $a$, $b$ and $a \cdot b$ are all contained in the compositum $\bar{k(x)} \cdot \bar{k(y)} =: M \subset K$.
  Since $x$ and $y$ are algebraically independent over $k$, we may identify
  \[ \Gal(M|k(x,y)^{i}) = \Gal(\bar{k(x)}|k(x)^{i}) \times \Gal(\bar{k(y)}|k(y)^{i}), \]
  where the two projections are the usual restriction maps.
  This lets us identify $\Gal(\bar{k(x)}|k(x)^{i})$ with the subgroup of $\Gal(M|k(x,y)^{i})$ which fixes $\bar{k(y)}$ pointwise, and similarly with $x$ and $y$ interchanged.

  With this identification, any $\sigma \in \Gal(\bar{k(x)}|k(x)^{i})$ acts on $a \cdot b$ as $\sigma(a) \cdot b$, and thus $\sigma(a)/a = \sigma(a \cdot b)/(a \cdot b) =: t_{\sigma}$.
  As $x$ and $x \cdot y$ are algebraically independent, we have
  \[ t_{\sigma} \in \bar{k(x)} \cap \bar{k(x \cdot y)} = k. \]
  But $\sigma(a) = t_{\sigma} \cdot a$ and $k(x)^{i}(a)$ is a finite extension of $k(x)^{i}$, hence $t_{\sigma}$ must be a root of unity.
  By symmetry, for any $\tau \in \Gal(\bar{k(y)}|k(y))$, we also have $\tau(b) = s_{\tau} \cdot b$ for some root of unity $s_{\tau}$.

  It follows that for all $\gamma \in \Gal(M|k(x,y)^{i})$, the elements $\gamma(a)/a$ and $\gamma(b)/b$ are both roots of unity.
  The action of $\Gal(M|k(x,y)^{i})$ on $k(x,y)^{i}(a,b)$ factors through a finite quotient, and thus we deduce that there exists some positive integer $n_{1}$ such that for all $\gamma \in \Gal(M|k(x,y)^{i})$, one has $\gamma(a^{n_{1}}) = a^{n_{1}}$ and $\gamma(b^{n_{1}}) = b^{n_{1}}$.
  In other words, $\Gal(M|k(x,y)^{i})$ acts trivially on $a^{n_{1}}$ and $b^{n_{1}}$, which implies that $a^{n_{1}} \in k(x)^{i}$, $b^{n_{1}} \in k(y)^{i}$.
  Arguing similarly with $x \cdot y$ in place of $y$ and $x^{-1}$ in place of $x$, we see that there exists an integer $n_{2}$ such that $a^{n_{2}} \in k(x)^{i}$ and $(a \cdot b)^{n_{2}} \in k(x \cdot y)^{i}$.
  Taking $n = n_{1} \cdot n_{2}$, it follows that $a^{n} \in k(x)^{i}$, $b^{n} \in k(y)^{i}$ and $(a \cdot b)^{n} \in k(x \cdot y)^{i}$.
  Further replacing $n$ by an integer of the form $n \cdot p^{k}$, where $p$ is the characteristic exponent of $k$, we may assume that $a^{n} \in k(x)$, $b^{n} \in k(y)$ and $(a \cdot b)^{n} \in k(x \cdot y)$.

  In particular, we may write
  \[ a^{n} = \prod_{c \in k} (x-c)^{m_{c}}, \ b^{n} = \prod_{c \in k} (y-c)^{n_{c}}, \]
  modulo constants, where all but finitely many of the $m_{c}, n_{c} \in \Zbb$ are zero.
  In particular, we must also have
  \[ (a \cdot b)^{n} = \prod_{c \in k} (x-c)^{m_{c}} \cdot (y-c)^{n_{c}}, \]
  modulo constants.
  But this element is contained in $k(x \cdot y)$, so the only irreducible polynomials from $k[x,y]$ that may appear in the factorization of $(a \cdot b)^{n}$ are of the form $(x \cdot y - c)$ for some $c \in k$.
  Combining these observations, we deduce that $m_{c} = n_{c} = 0$ for all $c \in k^{\times}$, and that $m_{0} = n_{0}$.
  In other words, there exists an integer $m$ such that $a^{n} = x^{m}$ and $b^{n} = y^{m}$ modulo constants, as required.
\end{proof}

Recall that $\Kcal_{K|k}$ is written \emph{additively}.
It will be useful to reinterpret the above lemma in the context of $\Kcal_{K|k}$.
First of all, if $S$ is \emph{any subset} of $\Kcal_{K|k}$, we write $\acl_{K|k}(S)$ for the \emph{smallest} element of $\Gbb_{\Kcal}(K|k)$ which contains $S$.
Since $\Gbb_{\Kcal}(K|k)$ is a complete lattice, this $\acl_{K|k}(S)$ is well-defined.
We say that two elements $x,y \in \Kcal_{K|k}$ are \emph{dependent} provided that there exists a geometric subspace of the form $\Kcal_{L|k}$ where $L \in \Gbb(K|k)$ with $\trdeg(L|k) \le 1$ such that $x,y \in \Kcal_{L|k}$.
If $x,y$ are \emph{not dependent} then we shall say that they are \emph{independent} and in this case $\acl_{K|k}(x,y) = \Kcal_{L|k}$ for some $L \in \Gbb(K|k)$ with $\trdeg(L|k) = 2$.
If $x,y \in K^{\times}$ with images $\bar x, \bar y$ in $\Kcal_{K|k}$, then $\bar x,\bar y$ are (in)dependent in the above sense if and only if $x,y$ are algebraically (in)dependent over $k$.

\begin{lemma}\label{lemma:acl_mul_compat_Kcal}
  Let $x,y \in \Kcal_{K|k}$ be two independent elements.
  Let $a,b \in \Kcal_{K|k}$ be two elements such that $a \in \acl_{K|k}(x)$, $b \in \acl_{K|k}(y)$ and $a + b \in \acl_{K|k}(x+y)$.
  Then there exists some $r \in \Qbb^{\times}$ such that $a = r \cdot x$ and $b = r \cdot y$.
\end{lemma}
\begin{proof}
  The map $\Kcal_{K|k} \to \Kcal_{\bar K|\bar k}$ is injective and compatible with the notion of (in)dependence introduced above.
  Also, the map
  \[ \KM_{1}(\bar K|\bar k) \to \Kcal_{\bar K|\bar k} \]
  is an \emph{isomorphism} since $\bar K^{\times}$ is divisible.
  By mapping into $\Kcal_{\bar K|\bar k}$ and identifying this with $\KM_{1}(\bar K|\bar k)$, the assertion of this lemma reduces to that of Lemma~\ref{lemma:acl_mul_compat}.
\end{proof}

\subsection{The relative case}

\begin{theorem}\label{theorem:main_relative_Kcal}
  Let $K|k$ be a relatively algebraically closed extension of fields satisfying $\trdeg(K|k) \geq 5$.
  Then $K^{i}|k^{i}$ are determined up-to isomorphism by the algebra $\Kcal_{*}(K|k)$.
  More precisely, if $L|l$ is another relatively algebraically closed extension of fields and
  \[ \phi_{*} : \Kcal_{*}(K|k) \xrightarrow{\cong} \Kcal_{*}(L|l) \]
  is an isomorphism of $\Qbb$-algebras, then there exists some $r \in \Qbb^{\times}$ and an isomorphism of fields
  \[ \Phi : K^{i} \xrightarrow{\cong} L^{i} \]
  such that $\Phi(k^{i}) = l^{i}$ and the following diagram commutes:
  \[ \begin{tikzcd}
      \Kcal_{K^{i}|k^{i}} \ar[r,"\Phi"] & \Kcal_{L^{i}|l^{i}} \\
      \Kcal_{K|k} \ar[u] \ar[r,"r \cdot \phi_{1}"'] & \Kcal_{L|l}. \ar[u]
    \end{tikzcd}\]
  Here, the vertical maps are those induced by the inclusions $K \subset K^{i}$ and $L \subset L^{i}$, and the map labeled $\Phi$ is induced by the isomorphism $\Phi : K^{i} \cong L^{i}$ of fields.
  Finally, $\Phi$ is unique with these properties up-to composition with some power of the $p$-power Frobenius $x \mapsto x^{p}$, where $p$ is the characteristic exponent of $K$.
\end{theorem}
\begin{proof}
  Let $\phi_{*}$ be as in the statement of the theorem and put $\phi = \phi_{1}$.
  By Lemma~\ref{lemma:trdeg_eq_of_iso}, we see that $\trdeg(L|l) = \trdeg(K|k) \geq 5$.
  By Theorem~\ref{theorem:recover_Kcal_geom_lattice} and the discussion around Fact~\ref{fact:alternating_iff_milnor_condition}, we see that the map $\Hcal \mapsto \phi \Hcal$ is an \emph{isomorphism} of lattices $\Gbb_{\Kcal}(K|k) \cong \Gbb_{\Kcal}(L|l)$.
  By Proposition~\ref{proposition:geometric_lattice_iso}, we obtain an isomorphism $\psi : \Gbb(K|k) \cong \Gbb(L|l)$ of geometric lattices satisfying
  \[ \phi \Kcal_{M|k} = \Kcal_{\psi M|l} \]
  for all $M \in \Gbb(K|k)$.
  Thus, by~\cite[theorem 4.2]{zbMATH05350217} (see Theorem~\ref{theorem:Evans-Hrushovski-Gismatullin}), there exists an isomorphism $\Phi : K^{i} \cong L^{i}$ of fields satisfying the following conditions:
  \begin{enumerate}
    \item $\Phi(k^{i}) = l^{i}$.
    \item One has $\psi(M) = \Phi(M^{i}) \cap L$ for all $M \in \Gbb(K|k)$.
  \end{enumerate}
  And furthermore, this $\Phi$ is unique up-to composition with some power of the $p$-power Frobenius $x \mapsto x^{p}$, where $p$ is the characteristic exponent of $K$.

  To conclude, we must show that there exists some $r \in \Qbb^{\times}$ making the diagram from the statement of the theorem commute.
  Let $\phi' : \Kcal_{K|k} \cong \Kcal_{L|l}$ be the isomorphism induced by $\Phi$, i.e.~$\phi'$ is the composition
  \[ \Kcal_{K|k} \xrightarrow{\cong} \Kcal_{K^{i}|k^{i}} \xrightarrow{\Phi} \Kcal_{L^{i}|l^{i}} \xleftarrow{\cong} \Kcal_{L|l}. \]
  We must show that there exists some $r \in \Qbb^{\times}$ such that $r \cdot \phi = \phi'$.

  Let $x,y \in \Kcal_{K|k}$ be two independent elements.
  The pair $\phi(x)$, $\phi(y)$ is also independent, while the following pairs are all \emph{dependent}:
  \[ \phi(x),\phi'(x); \ \phi(y),\phi'(y); \ \phi(x) + \phi(y),\phi'(x)+\phi'(y). \]
  By Lemma~\ref{lemma:acl_mul_compat_Kcal}, there exists some $r(x,y) \in \Qbb^{\times}$ such that
  \[ r(x,y) \cdot \phi(x) = \phi'(x), \ r(x,y) \cdot \phi(y) = \phi'(y). \]
  Since $\phi(x)$ and $\phi(y)$ are also \emph{linearly independent} over $\Qbb$, it follows that $r(x,y) = r(y,x)$.

  We claim that $r(x,y)$ does not depend on the choice of $x,y$.
  First, if $z$ is another element which is independent from $x$, then $r(x,y) = r(x,z)$ since $r(x,y) \cdot \phi(x) = r(x,z) \cdot \phi(x)$.
  Now, if $x',y'$ is any other pair of independent elements, we show that $r(x,y) = r(x',y')$ as follows:
  \begin{enumerate}
    \item If $x,x'$ are \emph{dependent}, then $x',y$ are independent, hence
          \[ r(x,y) = r(y,x) = r(y,x') = r(x',y) = r(x',y'). \]
    \item If $x,x'$ are \emph{independent}, then
          \[ r(x,y) = r(x,x') = r(x',x) = r(x',y'). \]
  \end{enumerate}
  In any case, $r(x,y)$ doesn't depend on the choice of $x,y$, and we put $r := r(x,y)$ for some choice of an independent pair $x,y$.

  Now, if $x$ is any nontrivial element of $\Kcal_{K|k}$, we may find some $y \in \Kcal_{K|k}$ which is independent from $x$, and observe that
  \[ r \cdot \phi(x) = r(x,y) \cdot \phi(x) = \phi'(x). \]
  Of course, this equality holds with $x = 0$ as well, so we deduce that indeed $r \cdot \phi = \phi'$.
  This concludes the proof of the theorem.
\end{proof}

\subsection{The absolute case}

\begin{theorem}\label{theorem:main_absolute_Kcal}
  Let $K$ be a field whose absolute transcendence degree is at least five.
  Then $K^{i}$ is determined up-to isomorphism by the algebra $\Kcal_{*}(K)$.
  More precisely, suppose that $L$ is another field and
  \[ \phi_{*} : \Kcal_{*}(K) \xrightarrow{\cong} \Kcal_{*}(L) \]
  is an isomorphism of $\Qbb$-algebras.
  Let $k$ be the relative algebraic closure of the prime subfield of $K$ and $l$ the relative algebraic closure of the prime subfield of $L$.
  Then $\phi_{*}$ induces an isomorphism $\bar\phi_{*} : \Kcal_{*}(K|k) \cong \Kcal_{*}(L|l)$.
  Furthermore, there exists some $r \in \Qbb^{\times}$ and an isomorphism of fields
  \[ \Phi : K^{i} \xrightarrow{\cong} L^{i} \]
  which necessarily satisfies $\Phi(k^{i}) = l^{i}$, such that the following diagram commutes:
  \[ \begin{tikzcd}
      \Kcal_{K^{i}|k^{i}} \ar[r,"\Phi"] & \Kcal_{L^{i}|l^{i}} \\
      \Kcal_{K|k} \ar[u] \ar[r,"r \cdot \bar\phi_{1}"'] & \Kcal_{L|l}. \ar[u]
    \end{tikzcd}\]
  Here, the vertical maps are those induced by the inclusions $K \subset K^{i}$ and $L \subset L^{i}$, and the map labeled $\Phi$ is induced by the isomorphism $\Phi : K^{i} \cong L^{i}$ of fields.
  Finally, $\Phi$ is unique with these properties up-to composition with some power of the $p$-power Frobenius $x \mapsto x^{p}$, where $p$ is the characteristic exponent of $K$.
\end{theorem}

By using Theorem~\ref{theorem:main_relative_Kcal}, in order to prove Theorem~\ref{theorem:main_absolute_Kcal} it suffices to give a characterization of the kernel of $\Kcal_{K} \to \Kcal_{K|k}$ using $\Kcal_{*}(K)$, where $k$ is the relative algebraic closure of the prime subfield.
If this kernel is denoted by $\Delta$, then one has
\[ \Kcal_{*}(K|k) = \frac{\Kcal_{*}(K)}{\langle \Delta \rangle} \]
so that we can then apply Theorem~\ref{theorem:main_relative_Kcal}.
We describe the characterization of this kernel in the following proposition.

For the rest of this subsection, we will use the notation $\Ecal^{\perp}$ to denote the orthogonal of a subspace $\Ecal \subset \Gcal_{K}$ with respect to the pairing
\[ \Kcal_{K} \times \Gcal_{K} \to \Qbb. \]
As before, it should be clear from context when $\Ecal^{\perp}$ refers to a subspace of $\Kcal_{K}$ as opposed to a subgroup of $K^{\times}$.

\begin{proposition}\label{proposition:prime_subfield_characterization}
  Let $K$ be any field, and let $k$ denote the relative algebraic closure of the prime subfield of $K$.
  Assume that $\trdeg(K|k) \geq 2$.
  Put $\Delta := \ker(\Kcal_{K} \to \Kcal_{K|k})$.
  For a nonzero $t \in \Kcal_{K}$, write $\Hcal_{t}$ for the Milnor closure of the subset $\{t\}$.
  Let $\Vscr_{t}$ denote the collection of closed subspaces $\Dcal \subset \Gcal_{K}$ satisfying the following conditions:
  \begin{enumerate}
    \item One has $\Zcal_{K}(\Dcal) \neq \Dcal$.
    \item One has $\Ccal_{K}(\Zcal_{K}(\Dcal)) = \Dcal$.
    \item One has $\Hcal_{t} \cap \Dcal^{\perp} = 0$.
  \end{enumerate}
  Then one has
  \[ \Delta = \bigcap_{t \in \Kcal_{K} \smin \{0\}} \bigcap_{\Dcal \in \Vscr_{t}} \Zcal_{K}(\Dcal)^{\perp}. \]
\end{proposition}
\begin{proof}
  Note that $\Hcal_{t}$ remains unchanged if we replace $t$ by $n \cdot t$ for any nonzero integer $n$.
  Thus, we may restrict our attention to those nonzero $t \in \Kcal_{K}$ which are in the image of $K^{\times}$.
  Note also that the kernel of $K^{\times} \to \Kcal_{K}$ is the torsion subgroup of $K^{\times}$.

  If $K$ has positive characteristic then $k^{\times}$ is the torsion of $K^{\times}$ hence $\Kcal_{*}(K) = \Kcal_{*}(K|k)$.
  Letting $t \in K \smin k$ be given, we deduce from Theorem~\ref{theorem:recover_Kcal_geom_lattice} that the intersection
  \[ \bigcap_{\Dcal \in \Vscr_{t}} \Zcal_{K}(\Dcal)^{\perp} \]
  is precisely the geometric subspace $\Kcal_{L|k}$ where $L$ is the relative algebraic closure of $k(t)$ in $K$.
  As $t$ varies, the intersection of all these geometric subspaces is trivial, so the assertion of the proposition follows.

  Assume for the rest of the proof that $K$ has characteristic zero.
  Let $\Dcal$ be a subspace satisfying conditions (1) and (2).
  By Theorem~\ref{theorem:detect_visible_valuations}, we see that $\Dcal = \Dcal_{v}$ for some visibile (i.e.~$\{1\}$-visible) valuation $v$ of $K$, and that $\Ical_{v} = \Zcal_{K}(\Dcal)$.

  Let $t$ be an element of $K^{\times}$ which is not a root of unity, and let $H_{t}$ denote the preimage of $\Hcal_{\bar t}$ in $K^{\times}$, where $\bar t$ denotes the image of $t$ in $\Kcal_{K}$.
  Assume that such a $\Dcal$ also satisfies condition (3) for this $\bar t$.
  We claim that the $v$ mentioned above must be trivial on $\Qbb$ hence also on $k$.
  For this, it suffices to show that $2 \in H_{t}$.
  Indeed, in this case we would find that $\{2,3,\ldots\} \subset H_{t}$ since $\Hcal_{\bar t}$ is Milnor-closed, using Lemma~\ref{lemma:milnor_closure_description}.
  Thus, if $v$ is nontrivial on $\Qbb$, and $p$ is the prime for which $v|_{\Qbb}$ is the $p$-adic valuation, we would have $1 + p \in H_{t} \cap \Urm_{v}^{1}$, while $\Dcal^{\perp} = \Ucal_{v}^{1}$, hence the image of $1+p$ would be a nontrivial element of $\Hcal_{\bar t} \cap \Dcal^{\perp}$.

  If $t^{n} = -2$ for some nonzero integer $n$, then we have $2 \in H_{t}$, so we may assume this is not the case.
  By replacing $t$ with $t^{n}$ for some nonzero integer $n$ (recall that this doesn't change $H_{t}$), we may assume that $1+t$ and $(2+t)/t$ are also not roots of unity (note that our assumptions ensure that $1+t^{n}$ and $(2+t^{n})/t^{n}$ are both nonzero for any nonzero integer $n$).
  Indeed, since $t$ is not a root of unity, there exists a complex embedding $\sigma : \Qbb(t) \to \Cbb$ such that $|\sigma(t)| \neq 1$, where $\Qbb(t)$ refers to the subfield of $K$ generated by $t$.
  Replacing $t$ with $t^{-1}$ if needed, assume that $|\sigma(t)| > 1$ hence $|\sigma(t^{n})| \to \infty$ as $n \to \infty$ and thus $1+t^{n}$ is not a root of unity for all sufficiently large $n$.
  Replace $t$ with such a $t^{n}$ with $n$ positive, and note that $t^{n \cdot k}$ would have also worked for any positive $k$.

  Now if $(2+t)/t$ is a root of unity, then $|2 + \sigma(t)| = |\sigma(t)|$ which forces $\sigma(t) = -1 + b \cdot i$ for some nonzero real $b$.
  If $(2+t^{2})/t^{2}$ is also a root of unity then $\sigma(t^{2})$ must have real part $-1$ as well, hence $-1 = 1 - b^{2}$ so that $b = \pm\sqrt{2}$.
  But in this case the real part of $\sigma(t^{3})$ is $5$, so that $(2 + t^{3})/t^{3}$ cannot be a root of unity.

  Having made this replacement, we now use Lemma~\ref{lemma:milnor_closure_description} repeatedly to find that $1 + t \in H_{t}$, $2 + t \in H_{t}$, $(2+t)/t \in H_{t}$ and $(2+t)/t-1 \in H_{t}$ as well.
  Thus, $2 = t \cdot ((2+t)/t - 1) \in H_{t}$.
  In any case, we have obtained that $2 \in H_{t}$, which implies that $v$ is trivial on $k$ as noted above.

  We deduce that $\Kcal_{k}$ is indeed contained in the intersection in question.
  On the other hand, if $t$ is transcendental over $k$, then we may extend $t$ to a transcendence base
  \[ \Bcal = \{t\} \cup \Bcal_{0} \]
  for $K|k$.
  Put $L := k(\Bcal_{0})$ and $M = k(\Bcal) = L(t)$.
  Let $v$ be an extension of the $t$-adic valuation on $M$ to $K$ and note that $v$ is visible by Proposition~\ref{proposition:visible_of_trdeg_ge_one} and our assumption on $\trdeg(K|k)$.
  However, the image of $t$ is not contained in $\Ucal_{v} = \Zcal_{K}(\Dcal_{v})^{\perp}$.
  On the other hand, letting $s \in \Bcal_{0}$ be any element (which is nonempty by our assumption on $\trdeg(K|k)$), we have $\Dcal_{v} \in \Vscr_{\bar s}$ where $\bar s$ is the image of $s$, by arguing as in the proof of Theorem~\ref{theorem:recover_Kcal_geom_lattice}.
  Thus, the image of $t$ is not contained in the intersection in question, and so the assertion of the proposition follows.
\end{proof}

\begin{proof}[Proof of Theorem~\ref{theorem:main_absolute_Kcal}]
  Let $k$ denote the relative algebraic closure of the prime subfield of $K$, and $l$ the relative algebraic closure of the prime subfield of $L$.
  The algebras $\Kcal_{*}(K)$ and $\Kcal_{*}(L)$ are isomorphic, while $K$ has absolute transcendence degree $\geq 5$.
  Such an isomorphism induces an isomorphism of structures $\Ascr(K) \cong \Ascr(L)$ by Fact~\ref{fact:alternating_iff_milnor_condition}, hence also an isomorphism $\Gcal_{K} \cong \Gcal_{L}$ which is compatible with alternating pairs.

  By Lemma~\ref{lemma:exists_subspace_of_trdeg_gt}, there exists a closed subspace $\Dcal \subset \Gcal_{K|k} \subset \Gcal_{K}$ such that $\Zcal_{K}(\Dcal) \neq \Dcal$ and $4 \le \dim(\Zcal_{K}(\Dcal))$.
  Transferring $\Dcal$ across $\Gcal_{K} \cong \Gcal_{L}$, we obtain a closed subspace $\Dcal' \subset \Gcal_{L}$ such that $\Zcal_{L}(\Dcal') \neq \Dcal'$ and $4 \le \dim(\Zcal_{L}(\Dcal'))$.
  By Proposition~\ref{proposition:centre_valuative}, $\Ical := \Zcal_{L}(\Dcal')$ is valuative, and if $v$ denotes the valuation associated to $\Ical$, it follows that $\Qbb \otimes_{\Zbb} vL$ has dimension $\geq 4$.
  Since $l$ the algebraic over a prime field, it follows that $\Qbb \otimes_{\Zbb} vL/vl$ has dimension $\geq 3$, and thus $3 \le \trdeg(L|l)$ by Abhyankar's inequality.

  In any case, we may thus apply Proposition~\ref{proposition:prime_subfield_characterization} to both $K$ and $L$.
  This shows that the isomorphism $\phi_{1} : \Kcal_{K} \cong \Kcal_{L}$ sends $\Kcal_{k}$ to $\Kcal_{l}$, and thus $\phi_{*}$ descends to an isomorphism
  \[ \bar\phi_{*} : \Kcal_{*}(K|k) = \frac{\Kcal_{*}(K)}{ \langle \Kcal_{k} \rangle} \cong \frac{\Kcal_{*}(L)}{\langle \Kcal_{l} \rangle} = \Kcal_{*}(L|l). \]
  By Theorem~\ref{theorem:main_relative_Kcal}, we obtain an isomorphism of fields $K^{i} \cong L^{i}$ and a rational number $r \in \Qbb^{\times}$ satisfying the conditions of our theorem.
\end{proof}

\subsection{The finitely-generated case}

In the case where $K|k$ is finitely generated, and $k$ is perfect, we use the work of \cite{zbMATH07463742} to obtain a better result.

\begin{theorem}\label{theorem:fg_case_relative}
  Let $k$ be a perfect field and $K$ a regular function field over $k$ of transcendence degree $\geq 2$.
  Then $K|k$ are determined, up-to isomorphism, from $\KM_{*}(K|k)$.
  More precisely, if $L|l$ is another regular function field over a perfect field of transcendence degree $\geq 2$ and $\phi_{*} : \KM_{*}(K|k) \cong \KM_{*}(L|l)$ is an isomorphism, then there exists an isomorphism $\Phi : K \cong L$ of fields satisfying $\Phi(k) = l$ and some $\epsilon \in \{\pm 1\}$ such that $\epsilon \cdot \phi_{1}$ is the isomorphism $K^{\times}/k^{\times} \cong L^{\times}/l^{\times}$ induced by $\Phi$.
  If $\trdeg(K|k) \geq 3$, then the assumption on $\trdeg(L|l)$ can be dropped.
\end{theorem}
\begin{proof}
  The isomorphism $\phi_{1} : K^{\times}/k^{\times} \cong L^{\times}/l^{\times}$ is compatible with algebraic dependence by Theorem~\ref{theorem:recover_Kcal_geom_lattice}.
  Thus the claim follows from~\cite[Theorem 4]{zbMATH07463742}.
  In the case where $\trdeg(K|k) \geq 3$, the fact that $\trdeg(L|l) \geq 2$ follows from Lemma~\ref{lemma:trdeg_eq_of_iso}.
\end{proof}

\begin{theorem}\label{theorem:fg_case_absolute}
  Let $K$ be a finitely-generated field whose absolute transcendence degree is at least two.
  Then the isomorphism type of $K$ is determined by $\KM_{*}(K)$.
  More precisely, suppose that $L$ is any other finitely-generated field of absolute transcendence degree $\geq 2$, $\phi_{*} : \KM_{*}(K) \cong \KM_{*}(L)$ is an isomorphism, $k$ denotes the relative algebraic closure of the prime subfield of $K$ and $l$ the relative algebraic closure of the prime subfield of $L$.
  Then $\phi_{*}$ induces an isomorphism
  \[ \bar\phi_{*} : \KM_{*}(K|k) \cong \KM_{*}(L|l), \]
  and there exists an isomorphism of fields $\Phi : K \cong L$ and some $\epsilon \in \{\pm 1\}$, such that $\epsilon \cdot \bar\phi_{1} : K^{\times}/k^{\times} \cong L^{\times}/l^{\times}$ is the isomorphism induced by $\Phi$.
  If $K$ has absolute transcendence degree $\geq 4$, then the assumption on the transcendence degree of $L$ can be dropped.
\end{theorem}
\begin{proof}
  Use the characterization of $\Kcal_{k}$ in $\Kcal_{K}$ from Proposition~\ref{proposition:prime_subfield_characterization} and the fact that the preimage of $\Kcal_{k}$ in $K^{\times}$ is $k^{\times}$ to observe that $\phi_{1} : K^{\times} \cong L^{\times}$ sends $k^{\times}$ to $l^{\times}$.
  Conclude by using Theorem~\ref{theorem:fg_case_relative}.
  If $\trdeg(K|k) \geq 4$, we may argue as in the proof of Theorem~\ref{theorem:main_absolute_Kcal} to see that $L$ has absolute transcendence degree $\geq 2$.
\end{proof}

\bibliographystyle{amsplain}
\bibliography{refs.bib}

\end{document}